\documentclass[11pt,a4paper, envcountsame]{amsart}
\usepackage{enumerate}
\def\url#1{{\tt #1}}
\usepackage[utf8]{inputenc}
\def\epsilon{\varepsilon}
\def\ge{\geq}
\def\le{\leq}

\usepackage{amsmath, amssymb, xspace}
\usepackage{graphicx}

\usepackage[all]{xy}

\newtheorem{theorem}{Theorem}[section]

\newtheorem{fact}[theorem]{Fact}

\newtheorem{proposition}[theorem]{Proposition}

\newtheorem{lemma}[theorem]{Lemma}

\newtheorem{question}[theorem]{Question}

\theoremstyle{definition}
\newtheorem{definition}[theorem]{Definition}
\newtheorem{example}[theorem]{Example}
\newcommand{\two}{\{0,1\}}

\newcommand{\NN}{{\mathbb{N}}}
\newcommand{\RR}{{\mathbb{R}}}
\newcommand{\R}{{\mathbb{R}}}
\newcommand{\QQ}{{\mathbb{Q}}}

\newcommand{\sub}{\subseteq}
\newcommand{\sN}[1]{_{#1\in \NN}}
\newcommand{\uhr}[1]{\! \upharpoonright_{#1}}
\newcommand{\ML}{Martin-L\"of}
\newcommand{\MLac}{ML absolutely continuous}

\newcommand{\SI}[1]{\Sigma^0_{#1}}

\newcommand{\bi}{\begin{itemize}}
\newcommand{\ei}{\end{itemize}}

\newcommand{\Halt}{{\ES'}}
\newcommand{\ES}{\emptyset}

\newcommand{\tp}[1]{2^{#1}}
\newcommand{\ex}{\exists}
\newcommand{\fa}{\forall}
\newcommand{\lep}{\le^+}
\newcommand{\gep}{\ge^+}

\newcommand{\seqcantor}{\two^{ \NN}}

\newcommand{\cantor}{\seqcantor}
\newcommand{\strcantor}{\two^{ < \omega}}

\newcommand{\Opcl}[1]{[#1]^\prec}

\newcommand{\n}{\noindent}

\newcommand{\leb}{\mathbf{\lambda}}

\newcommand{\sss}{\sigma}

\newcommand{\lland}{\, \land \, }

\newcommand \seq[1]{{\left\langle{#1}\right\rangle}}

\newcommand\+[1]{\mathcal{#1}}

\newcommand{\ape}{\hat{\ }}

\newcommand{\LR}{\Leftrightarrow}

\newcommand{\sssl}{\ensuremath{|\sigma|}}

\begin{document}
\title{Randomness and initial segment complexity for measures}
\author{Andr\'e Nies and Frank Stephan}
\thanks{Andr\'e~Nies is partially
supported by the Marsden Fund of the Royal Society of New Zealand,
13-UOA-184 and 19-UOA-346. Address: Department of Computer Science,
University of Auckland, Private Bag 92019, Auckland, New Zealand.
Email: andre@cs.auckland.ac.nz.}

\thanks{Frank~Stephan is partially supported by
supported by the Singapore Ministry of Education Academic
Research Fund Tier 2 grants MOE2016-T2-1-019 / R146-000-234-112 (2017--2020)
and MOE2019-T2-2-121 / R146-000-304-112 (2020--2023).
Address: Department of Mathematics and
Department of Computer Science,
National University of Singapore, 10 Lower Kent Ridge Road, Block S17,
Singapore 117690, Republic of Singapore.
Email: fstephan@comp.nus.edu.sg.}

\begin{abstract}
We study algorithmic randomness properties for probability measures on
Cantor space. We say that a measure
$\mu$ on the space of infinite bit sequences is \ML\ absolutely
continuous if the non-\ML-random bit sequences form a null set with
respect to~$\mu$.
We think of this as a weak randomness notion for measures. 

We begin with examples, and provide a robustness property related to Solovay
tests. The initial segment complexity of a measure~$\mu$ at   a length $n$ is defined as the  $\mu$-average over the
descriptive complexity of strings of
  length~$n$,  in the sense of either $C$ or~$K$. We relate this  weak randomness notion for a measure to the growth of its  
initial segment complexity. We show that a maximal growth implies the weak
randomness property, but also that both implications of the
Levin-Schnorr theorem fail. We 
discuss $C$-triviality and $K$-triviality for measures and relate
these two notions with each other. Here, triviality means that the
  initial segment complexity grows  as slowly as possible.

We show that full Martin-L\"of randomness of a measure  implies \ML\
absolute continuity; the converse fails because only the latter
property is compatible with having atoms.
In a final section we consider weak randomness relative to a general
ergodic computable measure. We seek appropriate effective
versions of the Shannon-McMillan-Breiman theorem and the Brudno
theorem where the bit sequences are replaced by measures.
We conclude with several open questions.
\end{abstract}

\maketitle
\tableofcontents
\section{Introduction}

\n The theory of algorithmic randomness is usually developed for bit
sequences, or equivalently, reals in the unit interval. A central randomness notion based on algorithmic tests is
the one due to \ML\ \cite{Martin-Lof:66}. 

Let $\cantor$ denote the  topological   space  of infinite  bit sequences. 
A probability measure $\mu$ on  $\cantor$ can be seen as a statistical
superposition of bit sequences. A single bit sequence   $Z$ forms an 
extreme case: the corresponding measure  $\mu$ is the  Dirac measure~$\delta_{Z}$,  i.e., $\mu$ is concentrated on   $\{Z\}$.  At  the opposite 
extreme is the uniform measure~$\leb$ which independently gives each 
bit value the probability $1/2$.  The uniform measure represents  
maximum disorder as no bit sequence is preferred over any other.

Recall that a measure $\mu$ on $\cantor$  is called  {absolutely
continuous} if each $\leb$-null set is a $\mu$-null  set. We introduce  \ML\  absolute continuity, an  algorithmic randomness notion for probability
measures that 
is a \emph{weakening}  of absolute continuity: we require that  the $\leb$-null set
in the hypothesis    be effective in the sense of \ML.  Given that
there is a universal \ML\ test, and hence a largest effective null
set,   all we have to require is that $\mu(\+ C) = 0$ where $\+ C$ is
the class of bit sequences that are not \ML\  random.
(By  a ``measure" we mean a probability measure on Cantor space unless otherwise
stated.)

Our research is motivated in part   by  a recent definition 
of \ML\ randomness for quantum states corresponding to
infinitely many qubits, due to the first author and
Scholz~\cite{Nies.Scholz:18}.  Let $M_n$ be  the algebra of  complex
$2^n \times 2^n$ matrices. Using the terminology there,  
measures correspond to the quantum
states $\rho$ where the matrix $\rho \uhr {M_n}$ is diagonal for each
$n$. For such diagonal states,  subsequent work of Tejas Bhojraj has shown  that 
the randomness notion defined there is equivalent to the one proposed
here for measures. So the measures form a useful intermediate case to test
conjectures in the   subtler  setting of quantum states.  This
applies, for instance, to the Shannon-McMillan-Breiman theorem discussed below, which is   studied in the setting of quantum states by the
first author and Tomamichel (see the post in~\cite{LogicBlog:17}).

\smallskip
\n \emph{Growth of the initial segment complexity.} Given a binary string
$x$, by $C(x)$ we denote its plain descriptive complexity, and by
$K(x)$ its prefix-free descriptive complexity. Some of our  motivation is
derived from the classical theory. Randomness of infinite bit
sequences is linked to the growth of the descriptive
complexity of their initial segments. For instance,    the
Levin-Schnorr theorem    intuitively says that randomness of $Z$ means
incompressibility 
(up to the same constant $b$)  of all the initial segments of $Z$. 
We want to study how much of this is retained in the setting of measures $\mu$. 
%We  now  take the  $\mu$-average over the complexity of all strings
%of a given length $n$. 
%
The  formal growth condition in the Levin-Schnorr theorem 
says that $K(x) \ge |x| -b$ for each initial segment $x$ of $Z$.  The ``$n$-th initial segment'' of a measure $\mu$ is given by
its values $\mu[x]$ for  all strings $x$ of length $n$, where $[x]$
denotes the set of infinite sequences extending~$x$. We  define the   complexities  $C(\mu \uhr n)$ and $K(\mu
\uhr n)$ of such an 
initial segment  as the $\mu$-average of the individual complexities
of the  strings of length~$n$.
With this definition, in Section~\ref{s:inseg} we show that both
implications of the  analog of the Levin-Schnorr theorem fail. 
However by  Proposition~\ref{prop: dim =1}  discussed below,   for 
measures  that are random in our weak sense, $C(\mu \uhr n) /n$, or
equivalently $K(\mu \uhr n) /n$, 
converges to $1$. Thus, such measures  have effective dimension~1;
see Downey and Hirschfeldt~\cite[Section 12.3]{Downey.Hirschfeldt:book}
for background on effective dimension. 

 It turns out that interesting new growth
behaviour is possible for measures. We   verify  in Fact~\ref{fast}
that the uniform measure $\leb$ has maximal  growth behaviour in the sense of~$K$, namely $K(\leb\uhr n)
\gep n+ K(n)$.  This implies that $\leb $ also has maximal growth in the sense of $C$. For bit sequences,  such a  growth rate is ruled out
 by a result of Katseff.  We   use    that,  in a sense to be made precise, 
most strings of a given length are incompressible.    We also provide an example of such a measure other than $\leb$.    On the other hand,  we show that  maximal  growth in the sense of $C$
implies ML absolute continuity.   

\smallskip
\n \emph{$K$-trivial and $C$-trivial measures.}
Opposite to random bit sequences lie  the $K$-trivial sequences $A$, where
the initial segment complexity grows no faster than that of a
computable set, in the sense that $K(A\uhr n) \lep K(n)$. Chaitin~\cite{Chaitin:76} proved that each $C$-trivial sequence is computable, while  Solovay~\cite{Solovay:75} showed that a $K$-trivial can be noncomputable. For background see e.g.\ \cite[Section
5.3]{Nies:book}.  In Section~\ref{s:Ktriv} we extend  these  notions to
statistical superpositions of bit sequences: we introduce  $C$- and $K$-trivial
measures. We first  show that such measures   have a countable support. This means
that they are countable convex combinations of Dirac measures (based on bit sequences that   necessarily satisfy the same notion of triviality). Thereafter, again guided by the case of bit sequences, we provide closure properties of the two notions: downward closure under truth-table reducibility, and closure under products.

 We leave the question  open whether each $C$-trivial measure is $K$-trivial. In Prop.\ \ref{pr:equivalence} we study three  conditions relating $C$-complexity and   $K$-complexity, and show that they are all equivalent.    If these conditions hold then each $C$-trivial measure is $K$-trivial. (Whether they  indeed hold was discussed at length in a working group at the 2020 American Institute of Mathematics workshop on randomness and applications. A~weaker version of such a  condition was considered in Bruno Bauwen's thesis~\cite{Bauwens:thesis}.) 
 
 We will also consider slightly stronger forms of  the notions of $C$- and $K$-triviality,  for which  this implication holds, and is proper.

\smallskip
\n \emph{\ML\ random  measures.}
Measures can be viewed as 
points in a canonical computable probability space, in the sense
of~\cite{Hoyrup.Rojas:09}. This yields a notion of  \ML\
randomness for measures. Bienvenu and Culver (see \cite[Theorem 2.7.1]{Culver:15})  have
shown that no measure $\mu$ that is \ML\ random is absolutely
continuous; in fact $\mu$ is orthogonal to $\lambda$ in the sense that
some $\leb$-null set is co-null with respect to $\mu$.  
 In contrast,  in Section~\ref{s:full ML} we show that this notion implies 
 our weak notion of randomness, \ML\ absolute continuity.  The  stronger
randomness notion forces the
measure to be  atomless, so the converse implication fails.  Further  
 questions can be asked  about the relationships between the different
randomness notions for measures we have discussed. For instance, does
the randomness notion studied by Culver imply  that the  initial
segment complexity for the measure (in the sense of $C$ or of $K$) is maximal infinitely often? 

\smallskip
\n \emph{Shannon-McMillan-Breiman Theorem.} This theorem from the 1950s says
informally that for an ergodic measure $\rho$ on $\cantor$, outside of an
appropriate null set, every bit sequence~$Z$ reflects the entropy of the
measure $\rho$ by the limit of the weighted information content w.r.t.\ $\rho$
of initial segments of $Z$, namely, $-\frac 1 n \log_2 \rho [Z\uhr n]$. See
e.g.\ \cite{Shields:96}, where the result is called the Entropy
Theorem. In the short Section~\ref{s:SMB}, we replace $Z$ by a measure
$\mu$ that is \ML\ a.c.\ with respect to~$\rho$, and take the $\mu$-average
of the information contents of strings of  the same length. We only obtain
a partial analog of the theorem. We use as a hypothesis that the weighted
information content of strings
w.r.t.\ $\rho$  is bounded, and also provide an example showing  that
this hypothesis is necessary.  However,  in a similar vein, in Proposition~\ref{prop: dim =1}  we 
establish an analog for measures of the effective Brudno's theorem
\cite{Hochman:09,Hoyrup:12}, which states that the entropy of $\rho$
is given as the
limit of $C(Z\uhr n)/n$, for any $Z$ that is $\rho$-ML-random. We show that one can take the limit of $C(\mu \uhr n)/n$ for any $\mu$ that is \MLac\  in $\rho$.

We conclude the paper  with  conclusions and  open questions.

For general background on recursion theory and algorithmic randomness we refer
the readers to the textbooks of Calude \cite{Calude:book},
Downey and Hirschfeldt \cite{Downey.Hirschfeldt:book},
Li and Vit\'anyi \cite{Li.Vitanyi:book},
Nies \cite{Nies:book}, Odifreddi \cite{Odifreddi:bookone,Odifreddi:booktwo}
and Soare \cite{Soare:book}.  
Lecture notes on recursion theory are also available online,
e.g.~\cite{Stephan:ln}. 

An 11-page conference version of the paper has appeared in STACS 2020~\cite{Nies.Stephan:20}. The present paper is  a much improved and expanded version, in particular
Sections~\ref{s:insegs growth} and \ref{s:Ktriv}.

%%%%%%%%%%%%%%%%%%%%
\section{Measures and Randomness} \label{sLMR}

\n In this section we will formally define one of our main notions,
\ML\ absolute continuity (Definition \ref{def:main}),
and collect some basic facts concerning it. In
particular, we  verify that  the well-known equivalence of \ML\ test
and Solovay tests extends to measures. 
We begin by briefly discussing algorithmic randomness for bit
sequences~\cite{Downey.Hirschfeldt:book,Nies:book}. We use
standard notation: letters $Z, X, \ldots$ denote elements of the space
of infinite bit sequences $\cantor$, $\sss,\tau $ denote finite bit
strings, and $[\sss] = \{ Z \colon \, Z \succ \sss\}$ is the set of
infinite bit sequences extending $\sss$.  $Z\uhr n$ denotes the string
consisting of the  first $n$ bits of $Z$. 
 For quantities $r,s$
depending on the same  parameters,  one  writes $r \lep s$ for $r \le
s +O(1)$.  A   subset $G$ of  $\cantor$ is called \emph{effectively open} if $G
= \bigcup_i[\sss_i]$ for a computable sequence $\seq{\sss_i}\sN i$ of strings. 
A measure $\rho$ on $\cantor$ is computable if the map $\strcantor \to
\R$ given by $\sss \mapsto \rho[\sss]$ is computable. That is, each 
real $\rho[\sss]$ is computable,  uniformly in $\sss$. 

\begin{definition} \label{def:ML}  Let $\rho$ be a computable measure
on $\cantor$.  A  $\rho$-\ML\ test  ($\rho$-ML-test, for short)
is a sequence $\seq{G_m} \sN m$ of uniformly effectively open sets such that
$\rho G_m \le \tp{-m}$ for each $m$. A bit sequence $Z$ \emph{fails}
the test if $Z \in \bigcap_m G_m$; otherwise it \emph{passes} the
test.  A bit sequence~$Z$ is $\rho$-\ML\ random ($\rho$-ML-random for
short) if $Z$
passes each $\rho$-\ML-test.
\end{definition}

\n  By  
$\leb$ one denotes  the uniform measure on~$\cantor$.  So $\leb [\sss]
= \tp{-\sssl}$ for each string~$\sss$. If no measure $\rho$ is
provided it will be tacitly assumed that $\rho = \leb$, and we will 
use the term ``\ML-random" instead of ``$\leb$-\ML-random" etc.
Let $K(x)$ denote the prefix free version of descriptive (i.e.,
Kolmogorov) complexity  of a bit  string $x$. 

\begin{theorem}[Levin \cite{Levin:73}, Schnorr \cite{Schnorr:73}] 
\label{thm:LS}  \

\n $Z$ is \ML-random $\LR$ $\ex b \, \fa n \, [K(Z\uhr n) \ge
n-b]$. \end{theorem}

\n Using the notation of \cite[Chapter 3]{Nies:book}, let 
 $\+ R_b$ denote  the set of bit sequences $Z$ such that 
$K(Z\uhr n) < n-b$ for some $n$. It is easy to see that $\seq {\+
R_b}\sN b$ forms a \ML\ test. 
 The Levin-Schnorr theorem says that this test is universal: $Z$ is
\ML-random iff it passes the test.

Recall that unless otherwise stated, all  measures are   probability
measures on Cantor space.
We use the letters $\mu, \nu, \rho$   for     measures (and
recall that $\lambda$ denotes the uniform measure).  We  now provide
the formal definition of  our weak randomness notion
for measures.

\begin{definition} \label{def:main} Let  $\rho $ be a measure.  
A measure $\mu$ is called
\emph{\ML\ absolutely continuous in $\rho$} if $\inf_m \mu(G_m)$ $=$ $0$
for each $\rho$-\ML-test $\seq{G_m}\sN m$. We denote this by 
$\mu \ll_{ML} \rho$. 

If $\inf_m \mu(G_m) = 0$ we say that $\mu$ \emph{passes} the test. 
If $\inf_m \mu(G_m) \ge \delta$ where $\delta >0$ we say $\mu$
\emph{fails the test at level} $\delta$.

In the case that $\rho= \leb$, we say that $\mu$ is \MLac, for short.
\end{definition}  

\n \ML\ absolute continuity is a \emph{weakening} of the usual notion that $\mu$ is 
absolutely  continuous in $\rho$, usually written $\mu \ll \rho$. In fact,  $\mu\ll\rho$ holds  iff
$\mu$ is $\rho$-$ML^X$ absolutely continuous for each oracle $X$. 

In the definition it suffices to consider $\rho$-ML tests $\seq{G_m}\sN m$
such that $G_m \supseteq G_{m+1}$ for each $m$, because we can replace
$\seq{G_m}$ by the $\rho$-ML test given by $\widehat G_m = \bigcup_{k >m} G_k$, and of
course $\inf_m \mu(\widehat G_m) = 0$ implies that $\inf_m \mu(G_m) = 0$. So we
can change the definition above, replacing the condition $\inf_m
G_m = 0$ by the only apparently stronger condition that $\lim_m G_m = 0$.

It is well known that there is a universal $\rho$-ML test. The intersection of all the components of such a   test consists of the
non-$\rho$-ML-random sequences. This implies:

\begin{fact} \label{fa:univ} $\mu \ll_{ML} \rho$  iff the
sequences that  are not $\rho$-ML-random form a $\mu$-null~set. \end{fact}

\n
We have already mentioned the two diametrically opposite types of examples of measure that are ML absolutely continuous in a measure $\rho$:

\begin{example}  (a) $\rho \ll_{ML} \rho$ for each measure $\rho$. 

(b) For a Dirac measure $\delta_Z$, we have $\delta_Z\ll_{ML} \rho$ 
iff $Z$ is $\rho$-ML-random.
\end{example}
Next we provide a source of non-examples where $\rho= \leb$.
\begin{example}  A Bernoulli measure on $\cantor$ assigns a fixed 
probability $p$ to a $0$, independently of the position. Such a
measure is not \MLac\ in the  case that $p\neq 1/2$. 
To see this, note that each \ML-random sequence $Z$ satisfies the law
of large numbers 
$$\lim_n \frac 1 n |\{ i< n \colon Z(i) =1\}| = 1/2;$$
see e.g.\ \cite[Proposition 3.2.19]{Nies:book}.
So if $\mu$ is \MLac, then $\mu$-almost
surely, $Z$ satisfies the law of large numbers. This is not the case
when $\mu$ is a  Bernoulli measures for $ p \neq 1/2$, because $\mu$ a.e.\ $Z$ satisfies the law of large numbers for $p$ instead of $1/2$. \end{example} 

\begin{definition} \label{def:local}
For a measure $\nu$ and string $\sss$ with $\nu[\sss] >0$,  let
$\nu_\sss$ denote  the localisation to $[\sss]$:
$$\nu_\sss(A) = \nu (A \cap [\sss])/ \nu[\sss].$$
Clearly if $\nu$ is \MLac\ then so is $\nu_\sss$.\end{definition}

A set $S$ of probability measures is called \emph{convex} if $\mu_i
\in S$ for $i \le k$ implies that the convex combination $\mu = \sum_i
\alpha_i \mu_i $ is in $ S$, where the $\alpha_i $ are reals in $
[0,1]$ and $\sum_i \alpha_i= 1$. The \emph{extreme points} of $S$ are the
ones that can only be written as convex combinations of length~1 of
elements of $S$. 

%If $\mu = O(\nu)$ and $\nu$ is \MLac\ then so is $\mu$. 

\begin{proposition} The \MLac\ probability
measures form a convex
set. Its extreme points are the \MLac\ Dirac measures, i.e.\ the
measures~$\delta_Z$ where $Z$ is a ML-random bit sequence.
\end{proposition} 

\begin{proof} Let $\mu = \sum_i \alpha_i \mu_i $
as above where the $\mu_i$ are \MLac\ measures. Suppose $\seq
{G_m}$ is a \ML\ test. Then
 $\lim_m \mu_i(G_m) = 0$ for each $i$, and hence $\lim_m \mu(G_m) =0$. 

Suppose that $\mu$ is \ML\
a.c. If $\mu$ is a Dirac measure then it is an extreme point of the \ML\
a.c.\ measures. Conversely, if $\mu$ is not Dirac, there is a least
number $t$ such that the decomposition $$\mu = \sum_{\sssl = t,
\mu[\sss] >0} \mu [\sss] \cdot \mu_\sss$$ is nontrivial. 
Hence $\mu$ is not an extreme point. 
\end{proof} 

\noindent
Recall that a $\rho$-\emph{Solovay-test}, for a computable 
probability measure $\rho$, is a sequence $\seq
{S_k}\sN k$ of uniformly $\SI 1$ sets such that $\sum_k \rho (S_k) <
\infty$. If $\rho = \leb$ we will simply use the term \emph{Solovay test}. 
A bit sequence $Z$ passes such a test if $Z \not \in S_k$ for almost
every $k$. (Each \ML-test is also a Solovay test, but  Solovay tests have a stronger  passing condition). A basic fact from the theory of
algorithmic randomness (e.g.\ \cite[Theorem 3.2.9]{Nies:book}) states that 
$Z$ is \ML-random iff $Z$ passes each Solovay test. This is usually
asserted only relative to the uniform measure $\leb$, but in fact carries
over to arbitrary measures~$\rho$.

The following characterises the \MLac\ measures with countable support.

\begin{fact} \label{fact:random_examples}
Let $\mu = \sum_k c_k \delta_{Z_k}$ where
$0 < c_k \leq 1$ for each $k$, and $\sum_k c_k = 1$.
Then $\mu$ is \MLac\ iff all the $Z_k$ are \ML\ random.
\end{fact}

\begin{proof} The implication from left to right is immediate. For
the converse implication, given a \ML\ test $\seq{G_m}$, note that
the $Z_k$ pass this test as a Solovay test. Hence for each~$r$, there
is $M$ such that for each $k \le r$ and each $m \ge M$ we have $Z_k \not \in G_m$ .
This implies that $\mu(G_m) \le \sum_{k>r} c_k$ for each $m \ge M$.
So $\lim_m \mu(G_m) =0$.
\end{proof}

\n
We say that a measure $\mu$ \emph{passes} a Solovay test $\seq
{S_k}\sN k$ if $\lim_k \mu(S_k) =0$.
The fact that passing all  \ML\ tests is  equivalent to passing all 
Solovay tests generalises from bit sequences to measures. 
We note that Tejas Bhojraj
(in preparation) proved  such a result in even greater generality in
the  setting of quantum states, where the proof is more involved.

\begin{proposition} For probability measures $\mu, \rho$ such that
$\rho$ is computable, we have

$\mu \ll_{ML} \rho$ iff $\mu$ passes
each $\rho$-Solovay-test. \end{proposition}

\begin{proof} Relative to $\rho$, each \ML\ test is a Solovay test,
and the passing condition $\lim_m \mu(G_m) =0$ works for both types of
tests by the remark after Definition~\ref{def:main}. This yields the
implication from right to left. 

For the implication from left to right, suppose that $\mu \ll_{ML}
\rho$, and let $\seq {S_k}\sN k$ be 
a $\rho$-Solovay-test. By $\limsup_k S_k $ one denotes the set of bit
sequences $Z$ such that $\ex^\infty k \,[Z \in S_k]$, that is, the
sequences that fail the test. By the basic fact (e.g.\
\cite[Theorem 3.2.9]{Nies:book}) mentioned above, the set of $\rho$-\ML-random
sequences is disjoint from $\limsup_k S_k $. Hence, by hypothesis on
$\mu$, we have $\mu(\limsup_k S_k)= 0 $. By Fatou's Lemma, $\limsup_k
\mu(S_k) \le \mu (\limsup_k S_k) $. So $\mu $ passes the Solovay test. 
\end{proof}

\section{A fast growing initial segment complexity implies being ML-a.c.}
\label{s:insegs growth}

\n
In this section we formally define our two variants of initial segment
complexity for measures, and    first their connections with ML absolute
continuity.

Recall that in this paper we use $C(x)$ to denote the plain
Kolmogorov complexity of a binary string $x$ and $K(x)$ to denote the
prefix-free Kolmogorov complexity.
Note that  other letters are also used in the literature; for example,
the textbook by Calude~\cite{Calude:book} uses $K$ and $H$, a notation
going back to Chaitin.

Recall the upper bounds  $C(x) \lep n$ and $K(x) \lep n + K(n)$
for all strings $x$ of length~$n$.
We assume for notational convenience the strict lower bounds
that $C(n) \leq C(x)$ and $K(n) \leq K(x)$. This convention
reduces the number of constants needed in the statements
and proofs of our results.
To achieve this, we replace some given universal machine $V$ by a new
universal machine $U$ such that whenever $V(p) = x$ then
$U(0p) = 0^{|x|}$ and $U(1p) = x$ with $0^n$ standing for $n$,
and $U$ is  undefined on all other inputs. It is easy to see
that if $V$ is universal by adjunction, so is $U$;   if
$V$ is prefix-free so is $U$. The descriptions for $U$ are only by  one bit
longer than those for $V$ and universality requires only that the
length of the shortest descriptions for some $x$ goes up at most by a constant.

\begin{definition} \label{def:ISC} Let
$K(\mu \uhr n):= \sum_{|x| =n}K(x) \mu[x] $ be the $\mu$-average of
all the $K(x)$ over all strings~$x$ of length $n$. In a similar
way  we define $C(\mu \uhr n)$. We also define the descriptive complexity
conditioned on $n$, namely $C(\mu \uhr n \mid n)$, as expected.
\end{definition}

\n It may help to think of $K(\mu \uhr n)$ as $\int  K_n d\mu_n$ where $K_n$ and $\mu_n$ are the respective functions restricted to the set of strings of length $n$.
Note that for each bit sequence $Z$, we have $K(\delta_Z \uhr n) = K(Z\uhr n)$.

In this section we tacitly rely on standard inequalities such as
$C(x) \lep K(x)$,
$K(x) \lep |x| +  2\log |x|$ and $K(0^n) \lep 2 \log n$. We also use
that for each~$r\in \NN$,  there are at most $2^r-1$ strings such that $C(x) <
r$. See e.g.\ \cite[Chapter~2]{Nies:book}. 

Note that for each string $x$ of length $n$,   the deviations from the upper bounds  are related by the inequality $n-C(x) \lep 2(n+ K(n)- K(x))$. For this, see  e.g.\ \cite[2.2.5]{Nies:book} which yields the inequality after some elementary algebraic manipulation. Taking the $\mu$-average,  this   implies the  corresponding inequality for the  initial segment complexity of  measures on Cantor space: 

\begin{fact} \label{fact: compare CK} Let $\mu$ be a measure. For each $n$ one has 
\[ n-C(\mu\uhr n) \lep 2(n+ K(n)- K(\mu \uhr n)). \]
\end{fact}

The following says that the uniform measure $\leb$ on $\cantor$ has the
fastest growing initial segment complexity that is possible in the sense
  of $K$, and therefore also of~$C$.

\begin{fact} \label{fast}  $K(\leb \uhr n) \gep n+K(n)$, and therefore also $C(\leb \uhr n) \gep n$.
\end{fact} 

\begin{proof}
Chaitin \cite{Chaitin:75} showed that there is a constant $c$ (dependent only 
on the universal prefix-free machine)
such that, for all $d$,  among the strings of length $n$, there are at most
$2^{n+c-d}$ strings with $K(x) \leq n+K(n)-d$; also see \cite[Thm.\ 2.2.26]{Nies:book}.

Fix $n$. For $d  \in \NN - \{0\}$  let $A_d$ be the set of strings of length $n$ such that $K(x) \le n + K(n) -d$.  Identifying sets of strings of the same length and the clopen sets they generate, by Chaitin's result we have $\leb A_d \le \tp{c-d}$.

Clearly  $K(x) =  n+ K(n) - \sum_{d >0} 1_{A_d}(x)$ whenever $|x| = n$. Taking the $\leb$-average over all strings of length $n$ we obtain
 
\begin{center}  $\tp{-n} \sum_{|x| = n} \ge n + K(n)  - \sum_{d> 0} \leb (A_d) \ge n + K(n) - \sum_{d>0}\tp{c-d} = n - \tp{c+1}$. \end{center} 
 \end{proof}

We provide an example of a computable probability measure $\mu\neq \leb$ that satisfies the hypothesis of Fact~\ref{fast}.  Informally, along a string, $\mu $ puts $1/3$ on $0$ until it hits the first~$1$;  from then on $\mu$ behaves like the  uniform measure~$\leb$. 

Recall from~\ref{def:local} on the localization $\nu_\sss$ of a measure $\nu$ to a string $\sss$. 
%
%that for a measure $\nu$ and string $\sss$ with $\eta[\sss] >0$, one defines the localization to $\sss$ by  $\eta_\sss (A)= \eta[\sss] ^{-1} \eta(A)$. 
For $r\in \NN -\{0\}$ let $\sss(r)$ denote the string $ 0^{r-1} 1$. Note that for a string $y$ of length $n-r$ one has $K(\sss(r) \ape  y)\gep K(y)$ (where the additive constant is independent of both $n$ and $r$). So by Fact~\ref{fast} for $n \ge r$ one has $K(\leb_{\sss(r)}) \gep n-r +K(n-r)$.   
%%%%%%%
\begin{fact}  The computable measure $\mu= 2\sum_{r>0} 3^{-r} \leb_{\sss(r)}$ satisfies $K(\mu \uhr n) \gep n+K(n)$. \\ Moreover,  $\mu$ is    absolutely continuous. \end{fact}

\begin{proof} \begin{eqnarray*} K(\mu \uhr n) &=& 2 \sum_{r=1}^\infty  3^{-r}  K(\leb_{\sss(r)})  \\
					&\gep &2 \sum_{r=1}^n  3^{-r}  (n-r +K(n-r)) \\
					&\gep &n -  2 \sum_{r=1}^n  r \cdot 3^{-r}    +2 \sum_{r=1}^n  3^{-r}  K(n-r) \\
					&\gep & n+K(n) \end{eqnarray*}
The last inequality holds because $2 \sum_{r=1}^\infty  r \cdot  3^{-r}  $ is finite, and because $K(n-r) \gep K(n)+ 2\log r$. 

Clearly $\mu$ is computable; $\mu$ is absolutely continuous because when viewed as a measure on the unit interval we have  $d\mu = g d\leb$ where the function $g\colon [0,1] \to \RR$ is given by  $g(0)= 0$,     $g(x) =  2\cdot (\frac 2 3)^r$ for $x$  in the interval $(\tp{-r-1}, \tp{-r}]$ where $r>0$ (this interval  corresponds to $\sss_r$).\end{proof}

\noindent
Li and Vitanyi  \cite{Li.Vitanyi:book} called a bit sequence $Z \in \cantor $  
  Kolmogorov random if there is $r$ such that $C(Z \uhr n) \ge
n-r$ for infinitely many $n$. One says that $Z$ is strongly Chaitin random
if there is $r$ such that  
 $K(Z \uhr n) \ge n+K(n)-r$ for infinitely many $n$. For bit sequences
these notions are 
equivalent to 2-random\-ness (i.e., ML-randomness relative to the halting problem) by \cite{Nies.Stephan.ea:05} and
\cite{Miller:10}, respectively;   also  see 
\cite[Theorem 8.1.14]{Nies:book} or \cite{Downey.Hirschfeldt:book}. 

One can  extend  these notions to measures. We call  $\mu$ strongly Chaitin random if  $K(\mu \uhr n) \gep n+K(n)$ for infinitely many $n$. Similarly we  define Kolmogorov randomness of a  measure. By Fact~\ref{fact: compare CK} the following is immediate:
\begin{fact} \label{fact:CK imply}  If a measure $\mu$ is strongly Chaitin random, then $\mu$ is Kolmogorov random. \end{fact}

We leave it open whether the reverse implication holds. The potentially weaker  notion  already   implies being \MLac:

\begin{theorem} \label{thm:Inseg MLR} Suppose that $\mu $ is a
Kolmogorov random measure.  Then $\mu $ is  \MLac. 
%(b) The same conclusion holds under the hypothesis that $K(\mu \uhr
%n) \ge n+K(n)-r$ for infinitely many $n$.
 \end{theorem}

\begin{proof} By hypothesis  there is   $r\in \NN$ such that  $C(\mu\uhr n) \ge n-r$ for
infinitely many $n$. Assume for a contradiction  that $\mu $ is not  \MLac. So
there is a \ML\ test $\seq {G_d} \sN d$ and $\delta > 0$ such that
$\mu (G_d) \ge   \delta$ for each $d$.
We  view $G_d$ as given by an enumeration of strings, uniformly in
$d$; thus $G_d = \bigcup_i [\sss_i]$  for a 
sequence $\seq {\sss_i}\sN i$ that is computable uniformly in~$d$. 
Let $G_d^{\le n}$ denote the clopen set
generated  by the strings in this enumeration of length at most $n$.
(Note that this set is not effectively given as a clopen set, but we
effectively have a  description of it as a $\SI 1$ set). Let $c$ be a constant
such that, for each~$x$ of length~$n$, one has  $C(x) \le n + c$. 

\begin{lemma} If $x$ is a string of length $n$ such that $[x] \sub
G_d^{\le n}$ then $C(x\mid  d) \lep n-d$. \end{lemma} 

\n
To  verify this, 
let $N$ be a fixed plain machine that on input $y$ and auxiliary input
$d$  prints out the $y$-th string  $x$ of length $n=|y| +d$ such that
the   enumeration of $G_d^{\le n}$ asserts that $[x] \sub G_d^{\le
n}$. (Here we view $y$ as the binary representation of a number, with
leading zeros allowed.)   Since $\leb G_d \le \tp{-d}$,  sufficiently
many strings are available to print all such $x$. This machine shows
that $C(x\mid  d) \lep n-d$ for any $x$ such that $[x] \sub G_d^{\le
n}$, as required. This verifies the lemma.

%For the case of prefix free complexity $K$, let $N'$ be the  slightly
%modified   machine where both $n$ and $d$ are auxiliary inputs.  The
%machine $N$ provides for a string $x$ of length $n$  a description of 
%length $n-d$. So for $N'$, for  the same pair $n,d$, the
%descriptions of  different strings form a  prefix free set. This
%verifies  the lemma.

Note that (after possibly increasing the constant $c$) any    $x$ as above satisfies  $C(x) \le n  - d + 2 \log d +c$.
% and $K(x) \le n + K(n) - d + 2 \log d  (after increasing $x$, if necessary)
%+c.$ % DISPLAY MODE KILLED TO GET 12 PAGES.
%
%We complete the proof separately for (a) and (b).
  For each $d,n$, letting $x$ range over strings of length $n$, we have 
\begin{equation*}C(\mu \uhr n) = \sum_{|x|=n} C(x) \mu[x] 
      =   \sum_{ [x] \sub G_d^{\le n}} C(x) \mu[x] + \sum_{ [x]
\not \sub G_d^{\le n}} C(x) \mu[x]. 
\end{equation*}
\n
The first summand is bounded above  by  $\mu(G_d^{\le n}) (n -d +
2 \log d +c )$ via the lemma, the second by $(1-\mu(G_d^{\le n})) (n +c)$.  We
obtain
 $$C(\mu \uhr n) \le n+ c - \mu(G_d^{\le n}) d/2.$$ 
Now for each $d$, for sufficiently large $n$ we have
$\mu(G_d^{\le n}) \ge  \delta$. So given $r$ letting,  $d = 2r/\delta$, for large enough $n$ we have $C(\mu\uhr n) \le n + c -r$. This contradicts the hypothesis on~$\mu$.  
%
%\smallskip
%\n  (b) For each $d,n$, letting $x$ range over strings of length $n$, we have 
%%
%\begin{equation*}K(\mu \uhr n)  =  \sum_{|x|=n} K(x) \mu[x] 
%       =    \sum_{ [x] \sub G_d^{\le n}} K(x) \mu[x] + \sum_{ [x]
%\not \sub G_d^{\le n}} K(x) \mu[x]. 
% \end{equation*}
%\n
%The first summand is bounded above  by  $\mu(G_d^{\le n}) (n+K(n) -d +
%2 \log d +c )$, the second by $(1-\mu(G_d^{\le n})) (n+K(n) +c)$. We obtain
%$$K(\mu \uhr n) \le n+K(n) + c - \mu(G_d^{\le n}) d/2.$$ 
%Now for each $d$, for sufficiently large $n$ we have
%$\mu(G_d^{\le n}) \ge  \delta$. As before, given $r$ let $d = 2r/\delta$; then
%for large enough $n$ we have $K(\mu\uhr n) \le n + K(n) + c -r$. 
\end{proof}

\n
It would be interesting to find out  whether the above-mentioned
coincidences of randomness notions for bit sequences lift to
measures; for instance, 
do the conditions in the theorem above actually imply that the measure is 
\MLac\ relative to the halting problem $\Halt$? 

\section{Both implications of the Levin-Schnorr Theorem fail for measures}
  \label{s:inseg}

\n
We will show that both implications of the analog of the Levin-Schnorr
Theorem~\ref{thm:LS} fail for measures.
The implication from left to right would say that a \MLac\ measure
cannot have an initial segment complexity in the sense of $K$ growing
slower than $n- O(1)$. This can be disproved by a simple
example of a measure with countable support.

Note that, by Proposition~\ref{prop: dim =1} below, we have
$\lim_n K(\mu\uhr n)/n=1$ for each \MLac\ measure $\mu$,
which provides a lower bound on the growth.

\begin{example} \label{ex:slow growth}
There is a \MLac\ measure $\mu$ such that for each $\theta
\in (0,1)$, one has $K(\mu\uhr n) \lep n-n^\theta$.
\end{example}

\begin{proof} We let $\mu = \sum_k c_k \delta_{Z_k}$ where $Z_k$ is \ML\
random and $0^{n_k} \prec Z_k$ for a sequence $\seq {c_k}\sN k $ of reals in
$[0,1]$ that add up to $1$, and a sufficiently fast growing computable sequence
$\seq{n_k}\sN k$ to be determined below. Then $\mu$ is \MLac\ by
Fact~\ref{fact:random_examples}.

For $n$ such that $n_k \le n < n_{k+1}$ we have
\begin{eqnarray*}
K(\mu \uhr n) & \lep & (\sum_{l=0}^k c_l) \cdot (n +2 \log n) +
(\sum_{l=k+1}^\infty c_l) \cdot 2 \log n \\
& \lep & (1-c_{k+1}) n + 2 \log n.
\end{eqnarray*}
Hence, to achieve $K(\mu\uhr n) \lep n - n^\theta$ it suffices to
ensure that
 $c_{k+1} n_k \ge n_{k+1} ^\theta + 2 \log n_{k+1}$
 for almost all $k$.
 For instance, we can let $c_k = \frac 1 {(k+1) (k+2)}$ and $n_k = \tp{k+4}$.
\end{proof}

\n To disprove the implication from right to left in a potential
generalisation to measures of the Levin-Schnorr Theorem, we need
to provide a measure
$\mu$ such that $\mu $ is not \ML\
a.c.\ yet  $K (\mu \uhr n) \gep n$.  This is  immediate from the following fact on the
growth of the initial segment complexity for certain bit sequences $X,Y$, letting $\mu = (\delta_X + \delta_Y)/2$.

\begin{theorem}
There are a \ML\ random bit sequence $X$ and a
non-\ML-random bit sequence $Y$ such that,
for all $n$, $K(X \uhr n)+K(Y \uhr n) \geq^+ 2n$.
\end{theorem}

\begin{proof}
Let $X$ be a low Martin-L\"of random set (i.e., $X' \equiv_T \ES'$).
We claim that there is a strictly increasing
function $f$ such that the complement of the range of $f$ is a recursively
enumerable set $E$, and $K(X \uhr m) \geq m+3n$
for all $m \geq f(n)$. To see this, recall that $\lim _n K(X \uhr n) -
n = \infty$. Since $X$ is low there is a computable function $p$
such that for all $n$, $\lim_s p(n,s)$ is the
maximal $m$ such that $K(X \uhr m) \leq m+3n$.

Define
$f(n,s)$ for $n \le s$ as follows. $f(n,0) = n$; for $s>0$ let $n$ be
least such that $p(n,s) \geq f(n,s-1)$ or $n=s$. If $m \geq n$ and
$n < s$ then let $f(m,s)=s+m-n$ else let $f(m,s)=f(m,s-1)$.

Note that for
each $n$ there are only finitely many $s>0$ with $f(n,s) \neq f(n,s-1)$
and that almost all $s$ satisfy $f(n,s) > p(n,s)$, as otherwise
$f(n,s)$ would be modified either at $n$ or some smaller value.
Furthermore, $f(n,s) \neq f(n,s-1)$ can only happen if there is
an $m \leq n$ with $f(m,s-1) \leq p(m,s)$ and that happens only
finitely often, as all the $p(m,s)$ converge to a fixed value and
every change of an $f(m,s)$ at some time $s$ leads to a value above $s$.
Furthermore, once an element is outside the range of $f$, it will
never return, and so the complement of the range of $f$
is recursively enumerable.
So $f(n) = \lim_s f(n,s)$ is a function as required, which
verifies the claim. (The complement $E$ of the range of $f$ is called
a Dekker deficiency set in the literature \cite{Odifreddi:bookone}.)

Now let $g(n) = \max\{m: f(m) \leq n\}$ (with the convention that
$\max (\ES) = 0$). Since $g$ is unbounded, by a result of Miller and
Yu~\cite[Corollary 3.2]{Miller.Yu:11} there is a Martin-L\"of
random $Z$ such that there exist infinitely many $n$ with
$K(Z \uhr n) \leq n+g(n)/2$; note that the result of Miller
and Yu does not make any effectivity requirements on $g$. Let
$$Y = \{n+g(n): n \in Z\}.$$
Note that $K(Z \uhr n) \leq K(Y \uhr n)+g(n)+K(g(n))$,
as one can enumerate the set $E$ until there are, up to $n$, only
$g(n)$ many places not enumerated and then one can reconstruct
$Z \uhr n$ from $Y \uhr n$ and
$g(n)$ and the last $g(n)$ bits of $Z\uhr n$. As $Z$ is Martin-L\"of random,
$K(Z \uhr n) \geq^+ n$, so
$$K(Y \uhr n) \geq^+ n-g(n)-K(g(n)) \geq^+ n-2g(n).$$
The definitions of $X,f,g$ give $K(X \uhr n) \geq n+3g(n)$.
This shows that, for almost all $n$, $K(X \uhr n)+K(Y \uhr n) \geq 2n$.

However, the set $Y$ is not Martin-L\"of random, because there are infinitely
many $n$ such that $K(Z \uhr n) \linebreak[3] \leq^+ \linebreak[3] n+g(n)/2$.
Now $Y \uhr {n+g(n)}$ can be computed from $Z \uhr n$
and $g(n)$, as one needs only to enumerate $E$ until the $g(n)$ nonelements
of $E$ below $n$ are found,  and they allow to see where the zeroes have to
be inserted into the string $Z \uhr n$ in order to obtain
$Y \uhr {n+g(n)}$. We have $K(g(n)) \leq g(n)/4$
for almost all $n$ and thus $K(Y \uhr {n+g(n)}) \leq^+ n+3/4 \cdot
g(n)$ for infinitely many $n$. So the set  $Y$ is not  Martin-L\"of random,
as required.
\end{proof}

\n \emph{Failing a strong Solovay test implies non-complex initial segments.}
\label{s:sS}
In contrast to the negative results above, we show that failing a stronger type of tests    leads to  considerable dips in the  intitial segment complexity of a measure in the sense of~$C$ (and hence, also in the sense of $K$ by Fact~\ref{fact: compare CK}).  

We say that a Solovay test $\seq
{S_r}\sN r $ is \emph{strong} if each $S_r$ is clopen and there is a recursive
function $g$ such that $g(r)$ is a
strong index for a finite set of strings $X_r$ such that $\Opcl {X_r}=
S_r$. For bit sequences, this means no restriction of the corresponding
randomness notion: in that case any Solovay test $\seq {G_m}$ can be replaced by
a strong Solovay test, listing strings that make up the uniformly $\SI 1$ sets $G_m$
one-by-one. We conjecture that this equivalence of
test notions no longer holds for measures.

The following is a weak version for measures of one implication of
the Miller-Yu theorem~\cite[Theorem 7.1]{Miller.Yu:08}.
We show that if a measure $\mu$ is far from \MLac, then its initial
segment complexity $C(\mu\uhr n | n)$ defined  in Definition \ref{def:ISC}
has infinitely many dips of size $\delta f(n)$ for some positive constant
$\delta$ and a computable function $f$ that grows somewhat fast in
that $\sum_n \tp{-f(n)} < \infty$. We use elements of its  proof
in Bienvenu, Merkle and Shen~\cite{Bienvenu.Merkle.ea:08}. We note that the
result is a variation on (the contrapositive of)
Theorem~\ref{thm:Inseg MLR}(a), proving a stronger conclusion from a
stronger hypothesis. A version of the result for quantum states is \cite[Thm.\ 4.4]{Nies.Scholz:18}.

\begin{proposition} Suppose that $\mu$ fails a strong Solovay test $\seq
{S_r}\sN r $ at level $\delta$, namely $\ex^\infty r \, [ \mu(S_r) \ge
\delta]$. Then there is a computable function $f$ such that $\sum_n
\tp{-f(n)} < \infty$ and
$$\ex^\infty n  \,  [C(\mu\uhr n \mid n) \lep n - \delta f(n)].$$ 
\end{proposition}

\begin{proof} Let $\seq{X_r}$ be as in the definition of a strong Solovay test.
We may assume that $\sum \leb S_r \le 1/2$, and all strings in $X_r$
have the same length $n_r$, where $n_r < n_{r+1}$ for each~$r$. We write
$\mu (X_r)$ for $\mu \Opcl{X_r}$. Let $f$ be the computable function
such that
$$2^{-f(n_r)} \ge \mu  (X_r) > 2^{-f(n_r)-1}$$
and $f(m) = m$ for each $m$ that is not of the form $n_r$. There is a
constant $d$ such that each bit string $x$ in the set $X_r$ satisfies
(where $n= n_r$)
\begin{equation}\label{eqn:compr} C(x \mid n)
\le n - f(n) + d=: g(r).\end{equation}
For, $r$ can be computed from $n=n_r$, and each string $x \in X_r$ is
determined by $r$ and its position $i< \tp n \leb (X_r) $ in the
lexicographical listing of $X_r$. We can determine $i$ by $\log
(\tp n \leb (X_r)) \le n - f(n) +d $ bits for some fixed $d$.
In fact we may assume the description has exactly that many bits.
Thus, there is a Turing machine $L$ with two inputs such that
for each $\sss \in X _{r}$, we have $L(v_\sss ; n) = \sss$ for some
bit string $v_\sss$ of length $g(r)$. 

Let $c$ be a constant such that $C(x) \le |x| + c$ for each string
$x$. Now suppose that $\mu (X_r) \ge \delta$. Then for $n= n_r$,
where $\sss$ ranges over strings of length $n$,
\begin{eqnarray*} C(\mu\uhr n \mid n) & \lep & \sum_{\sss \in X_r}
(n-f(n)) \mu[\sss] + \sum_{\sss \not \in X_r} (n+c) \mu[\sss] \\
   & \lep & n - f(n) \sum_{\sss \in X_r} \mu[\sss] \\
   & \lep & n- \delta f(n). \end{eqnarray*}
There are infinitely many such $r$ by hypothesis. This
completes the proof. 
\end{proof}

\section{$K$- and $C$-triviality for measures} \label{s:Ktriv}

\n
Recall from  the third paragraph of Section~\ref{s:insegs growth}
that our choice of universal machines implies
that $C(n) \leq C(x)$ and $K(n) \leq K(x)$ for all strings $x$ of length $n$.

\begin{definition} A measure $\mu$ is called $K$-trivial if $K(\mu
\uhr n) \lep K(n)$ for each $n$;

\n $\mu$ is called
$C$-trivial if $C(\mu \uhr n) \lep C(n)$ for all $n$.\end{definition}

\n Suppose $\mu$ is a Dirac measures $\delta_A$.
Then $\mu$ is $K$-trivial iff $A$  
is $K$-trivial in the usual sense, and $\mu$ is $C$-trivial iff $A$ is
$C$ trivial in the usual sense, which is equivalent to being recursive
by a well-known result of Chaitin~\cite{Chaitin:76} (for a more recent proof see e.g.\ \cite[5.2.20]{Nies:book}). More generally, any finite convex
combination of $K$-trivial Dirac measures is $K$-trivial, and similarly
for $C$. Now if $\alpha,\beta$ are
any positive real numbers such that $\alpha+\beta=1$,
then $\alpha \cdot \delta_\emptyset + \beta \cdot \delta_{\mathbb N}$
is a $C$-trivial and $K$-trivial measure such that any measure representation
in the sense of \cite[1.9.2]{Nies:book}
computes $\alpha$.
  Thus one cannot bound the computational complexity  
of $K$-trivial or of $C$-trivial measures in any way.

For a further example, let $\seq {R_i}$ be a sequence of uniformly recursive sets, and let $\seq {\alpha_i}$ be a sequence of uniformly recursive positive real numbers such that $\sum_i \alpha_i=1$ and $S=\sum_i \alpha_i \log i$ is finite. Let $\mu = \sum_i \alpha_i \delta_{R_i}$. Since $K(R_i\uhr n) \lep K(n) + K(i)$ and $K(i) \lep 2\log i$, for each $n$ we have $K(\mu\uhr n) =  \sum_i \alpha_i K(R_i\uhr n)\lep K(n) + 2S$. Hence $\mu$ is $K$-trivial.  

\smallskip
\n
{\em Atoms and Triviality.} We show that  both $K$-trivial and $C$-trivial
measures are discrete in the sense of \cite{Schnorr:77};  namely,
they are concentrated on the set of their atoms.  
\iffalse
%A survey of the theory of random sequences. In: Basic problems in methodology and linguistics (Proceedings Fifth Internat. Congr. Logic, Methodology and Philos. of Sci., Part III, Univ. Western Ontario, London, Ont., 1975), University Western Ontario Ser. Philos. Sci., vol. 11,
%pp. 193?211. Reidel, Dordrecht (1977)
\fi
Such measures were further studied by Kautz \cite{Kautz:thesis}, where they were
called ``trivial measures'', and more recently by Porter \cite{Porter:15}.
%Theory Comput Syst (2015) 56:487?512 DOI 10.1007/s00224-015-9614-8

\begin{proposition}\label{prop:atomos supportos}
If a measure $\mu$ is $K$-trivial, then $\mu$ is
supported by its set of atoms.

\n
In fact the weaker hypothesis $ \ex p \, \ex^\infty n \, [K(\mu \uhr
n) \le K(n)+p]$ suffices for this.  A similar statement holds for $C$.
\end{proposition}

\begin{proof} For a set $R \sub \{0,1\}^*$, by $\Opcl R$ one denotes
the open set $\{Z \colon \, \ex n \, Z\uhr n \in R]\}$.

Assume for a contradiction that $\mu$ gives a measure
greater than $\epsilon >0$ to the set of its non-atoms. Fix $c$
arbitrary with the goal of showing that $K(\mu \uhr n) \ge K(n) +
\epsilon c /2$ for large enough $n$. 

There is a constant $d$ (in fact $d = O(2^c)$) such that for each $n$
there are at most $d$ strings~$x$ of length $n$ with $K(x) \le K(n) + c$
(see e.g.\ \cite[Theorem 2.2.26]{Nies:book}).

Let $S_n = \{x\colon \, |x| = n \lland \mu [x] \le \epsilon  /2d \}$.
By hypothesis we have $\mu \Opcl{S_n} \ge \epsilon$ for large enough
$n$. Therefore by choice of $d$ we have \[\mu \Opcl { S_n \cap \{x
\colon \, K(x) > K( |x|) + c\}} \ge \epsilon/2. \]
Now we can give a lower bound for the $\mu$-average of $K(x)$
over all strings $x$ of length $n$: 
\[ \sum_{|x| =n} K(x) \mu[x] \ge (1- \frac \epsilon 2 K(n) +
\frac \epsilon 2 (K(n) +c) \ge K(n) + \epsilon c/2, \]
as required. Notice that we have only used the weaker hypothesis. 
\end{proof} 

\noindent
Thus, if   $\mu$ is $K$-trivial for constant $p$, then  $\mu $ has the
form $ \sum_{r< N}   \alpha_r \delta_{A_r}$ where $N \le \infty$ and 
each $\alpha_r$ is positive and $\sum_{r< N} \alpha_r =1$. The following 
observation  will show that in addition all the $A_r$ are $K$-trivial. In
particular, $K$-trivial measures are not Martin-L\"of almost continuous.

\begin{fact} Suppose $\mu = \sum_r \alpha_r \delta_{A_r}$ is $K$-trivial
for constant $p$. Then each $A_r$ is $K$-trivial for the constant $p/\alpha_r$.
Similarly, if $\mu$  is $C$-trivial with
constant $p$, then each $A_r$ is $C$-trivial with constant $p/\alpha_r$.
\end{fact}

\begin{proof} We have 
\[ K(n) +p  \ge K(\mu \uhr n) = \sum_s \alpha_s K(A_s\uhr n) \ge 
   K(n) + \alpha_r(K(A_r \uhr n)-K(n)). \]
Therefore, $K(A_r \uhr n) - K(n) \le p/\alpha_r$, as required.
The proof for $C$ is obtained via replacing $K$ by $C$ everywhere.
\end{proof}
\n
Above we built a  computable $K$-trivial measure with
infinitely many atoms.
% THIS goes nowhere
% as follows. Let $A_r = 0^{r+1} 1^\infty$, so
%that $K(A_r\uhr n ) \lep K(n) + 2 \log r$. Let $\mu = \sum \tp{-r+1}
%A_r$. By the above fact $\mu$ is $K$-trivial.  If we vary the
%construction by letting $A_r = 0^{r+1} 1 B$ where $B$ is $K$-trivial
%but non-recursive, we obtain a $K$-trivial measure $\mu$ with infinitely
%many atoms and none of them is recursive.
%
%Actually there is even a $K$-trivial measure such that all $K$-trivial
%sets are atoms of this measure, but with a sequence of coefficients which
%very fast tends to $0$.
%
On the other hand, the following example shows that not every 
infinite convex combination $\mu = \sum_k \alpha_k \delta_{A_k}$ of
$K$-trivial Dirac measures for constants $b_k$ yields a
$K$-trivial measure, even if there is a finite bound on the values
$\alpha_k  b_k$. Let $A_k = \{\ell: \ell \in \Omega \wedge \ell < k\}$ and
$\alpha_k = (k+1)^{-1/2}-(k+2)^{-1/2}$. All sets $A_k$ are finite and
thus $K$-trivial for constant $2k+ O(1)$. Furthermore, the
sum of all $\alpha_k$ is $1$ and $\alpha_k = O(1/k)$.
We have 
\[K(\mu\uhr n) = \sum_{|x|=n} K(x) \mu(x)   
     \geq (\sum_{m \geq n} \alpha_m) \cdot K(\Omega\upharpoonright n)
     \geq (n+2)^{-1/2} \cdot (n+2) = \sqrt{n+2} \]
for almost all $n$, and thus the average grows faster than $K(n)+c$.
So the measure is not $K$-trivial.

In a sense, an 
atomless measure can be  arbitrarily close to
being $C$-trivial and $K$-trivial. 

\begin{proposition} 
For each nondecreasing unbounded function $f$ which is recursively
approximable from above there is a non-atomic measure $\mu$ such that
$C(\mu \uhr n) \leq^+ C(n)+f(n)$ and $K(\mu \uhr n) \leq^+ K(n)+f(n)$.
\end{proposition}

\begin{proof}
The proof for $C$ and $K$ is essentially the same; we present  
the one for $K$ first.

There is a recursively enumerable
set $A$ such that, for all $n$, $A\cap \{0, \ldots, n\}$ has up to a
constant $f(n)/2$ non-elements. One lets $\mu$ be the measure such that
$\mu(x) = 2^{-m}$ in the case that all ones in $x$ are not in $A$
and $\mu(x) = 0$ otherwise, where $m$ is the number of non-elements
of $A$ below $|x|$. One can see that when $\mu(x) = 2^{-m}$
then $x$ can be computed from
$|x|$ and the string $b_0 b_1 \ldots b_{m-1}$ which describes the
bits at the non-elements of $A$. Thus 
\[K(x) \leq^+ K(|x|)+K(b_0 b_1 \ldots b_{m-1}) \leq^+ K(|x|)+2m.\]
It follows that $K(\mu \uhr n) \leq^+ K(n)+f(n)$,
as the $\mu$-average of strings $x \in \{0,1\}^n$
with $K(x) \leq^+ K(n)+f(n)$ is at most $K(n)+f(n)$ plus a constant.

Now the bound for $C$ follows along the same lines;
note that for $C$ one gives a prefix-free coding of the string
$b_0 b_1 \ldots b_{m-1}$ above followed by a $C$-description for $|x|$
plus a constant-length coding prefix -- assuming that the universal machine is
universal by adjunction. Again the length of the code is $C(n)+f(n)$
plus a constant.
\end{proof}

 \smallskip
\n
{\em  Closure properties of $C$- and of  $K$-triviality.}
In the following suppose $\Gamma$ is a total Turing reduction,
that is, a truth-table reduction, mapping  $\cantor$ to $\cantor$.
Given a measure~$\nu$, let $\mu=\Gamma(\nu)$ denote the image measure,
defined as usual by $\mu({\mathcal A}) =\nu(\Gamma^{-1}({\mathcal A}))$,
where $\mathcal A$ is a subset of $\cantor$.
As $\Gamma$ is a truth-table reduction, some  strictly increasing
computable function $f$   bounds the use of $\Gamma$.

\begin{proposition} \label{prop:gamma}
Let $\mu=\Gamma(\nu)$ as above.
We have $K(\mu\uhr n) \le^+ K(\nu \uhr{f(n)})$ and similarly for $C$. 
In particular, if $\nu$ is $K$-trivial then so is $\mu$, and if
$\nu$ is $C$-trivial so is $\mu$. \end{proposition}

\begin{proof} Let $x$ range over strings of length $n$ and $y$ over
strings of length $f(n)$. Since $K(\Gamma^y) \le K(y)$ we have
\begin{eqnarray*} K(\mu\uhr n) &=& \sum_{x} K(x) \sum_{
\Gamma^y\succeq x} \nu[y]
\ = \  \sum_{x} \sum_{ \Gamma^y\succeq x} K(x) \nu[y] \\
&\le^+ & \sum_{x} \sum_{\Gamma^y\succeq x} K(y) \nu[y] 
\ = \  K(\nu \uhr{f(n)}). \end{eqnarray*}
Using that $n$ and $f(n)$ have, up to a constant, the same
descriptive     complexity in the sense of $K$, this completes the
proof. The proof for
$C$ is similar to   the proof for~$K$ with only the obvious changes.
\end{proof}

\n
Note that $\cantor \times \cantor$ is effectively isomorphic to
$\cantor$ via the map $A,B \to A\oplus B$. So a product of two
measures on $\cantor$ can itself be viewed as a measure on $\cantor$.
We show that $K$- and $C$-triviality of measures is closed under taking products,
generalising  well-known facts for bit sequences. The following fact will be needed for the case of~$C$. 
\begin{fact} \label{fa:upgrade}
$K( x \mid \, n, C(n)) \lep 2(C(x)-C(n))$ for each $n$ and each $  x\in \{0,1\}^n$. \end{fact}

\begin{proof}
One suffices  to show that there is a constant $d$ such that for each $b$, 
$$
     C(x) \leq C(n)+b \rightarrow
   K(x|n,C(n)) \leq 2b+d.
$$
Following an argument of Chaitin \cite{Chaitin:76},
Nies \cite[Proof of Lemma 5.2.21]{Nies:book} showed that
for every plain machine $M$ there are at most $O(b^2 \cdot 2^b)$ descriptions
$p$ with $M(p) = n$ and $|p| \leq C(n)+b$. Consider the machine $M$ such that  $M(p)=0^{|U(p)|}$ where $U$ is the universal plain machine.
 There are at most $O(b^2 \cdot 2^b)$ many
$x \in \{0,1\}^n$ with $C(x) \leq C(n)+b$, because any such  $x$ has the form  $U(p)$
for some  $p$ with $M(p)=n$ and $|p| \leq C(n)+b$. 
Uniformly in $b$,  $n$ and $C(n)$,  the set of such $x$ is recursively enumerable. Hence  one can provide a  prefix-free
code of length up to $2b$ plus constant
for any pair $\langle b,a\rangle$ with $a \in O(b^2 \cdot 2^b)$
selecting the $a$-th such $x$. This shows that   a constant $d$  as above exists. \end{proof}

\begin{theorem} \label{pr:sixfive}
Let $\mu$ and $\nu$ be measures.\\
(a) If $\mu$ and $\nu$ are $K$-trivial then $\mu \times \nu$ is $K$-trivial.\\
(b) If $\mu$ and $\nu$ are $C$-trivial then $\mu \times \nu$ is $C$-trivial.
\end{theorem}

\begin{proof}
We assume in both cases (a) and (b) that $\mu$ and $\nu$ are
trivial with respect to the corresponding descriptive string complexity and then show
that $\mu \times \nu$ is trivial for the same complexity notion.

For (a), by Prop~\ref{prop:atomos supportos} one may assume that
$\mu = \sum_i \alpha_i A_i$ and $\nu = \sum_j \beta_j B_j$ with
$1 = \sum_i \alpha_i$ and $1 = \sum_j \beta_j$, where the $A_i$ and $B_j$
are $K$-trivial sets. For a string $x$ of length~$n$ and
a shortest description $p$ of (a code for) $n$, $K(x) =^+ K(x|p)+|p|$;
this follows easily from~\cite[Theorem 3.9.1]{Li.Vitanyi:book}.
Note that
$$
   K(A_i \oplus B_j \uhr {2n}) \leq^+
   K(n) + K(A_i \uhr n | p)+K(B_j \uhr n | p).
$$
Since $\mu$ and $\nu$ are $K$-trivial we have
$K(A_i\uhr n)=^+ K(n)+K(A_i \uhr n|p)$
and $K(B_j\uhr n)=^+ K(n)+K(B_j \uhr n|p)$, so
 $\sum_i \alpha_i K(A_i \uhr n|p)$ and $\sum_j \beta_j K(B_j \uhr n|p)$
are both bounded by some constant independent of $n$.
Let $c$ be a common upper bound of such constants. Now, for some constant $d$,
\begin{eqnarray*}
   K(\mu \times \nu \uhr {2n})
    & = & \sum_{i,j} \alpha_i \cdot \beta_j \cdot K(A_i \oplus B_j \uhr {2n}) \\
    & \leq & \sum_{i,j} \alpha_i \cdot \beta_j \cdot
              (K(n) + K(A_i \uhr n | p)+K(B_j \uhr n | p)+d) \\
    & = & K(n)+d + (\sum_i \alpha_i \cdot K(A_i \uhr n|p) \cdot
         (\sum_j \beta_j)\\
    & & \mbox{ } \ \ \ \ \ \ \ \ + \ (\sum_i \alpha_i) \cdot
                         (\sum_j \beta_j \cdot K(B_j \uhr n|p) \\
    & \leq & K(n)+2c+d.
\end{eqnarray*}
Since $K(n) =^+ K(2n)$, one can conclude that
$
   \forall n\,[K(\mu \times \nu \uhr {2n}) \leq^+ K(2n)].
$
The same can be proven for $2n+1$ with only a slight more notational
complexity. Thus $\mu \times \nu$ is $K$-trivial.

%\iffalse
%For (b), note that by Fact~\ref{fa:upgrade} there is a constant $d$ with
%$$
%   \forall n\,\forall b\,\forall x\in \{0,1\}^n \mbox{ with } C(x) \leq C(n)+b
%   \ [K(x|n,C(n)) \leq 2b+d].
%$$
%Note that the shortest $C$-description of $n$ has length $C(n)$
%and codes both, $n$ and $C(n)$.
%\fi

For (b), as in (a), assume that
$\mu = \sum_i \alpha_i A_i$ and $\nu = \sum_j \beta_j B_j$ with
$1 = \sum_i \alpha_i$ and $1 = \sum_j \beta_j$,
where now $A_i,B_j$ are $C$-trivial sets.
Let $c$ be the supremum over $n$ of all $\sum_i \alpha_i \cdot (C(A_i \uhr n)-C(n))$
and $\sum_j \beta_j \cdot (C(B_j \uhr n)-C(n))$. Note that $c<\infty$
because  $\mu$ and $\nu$ are  $C$-trivial.

One can produce  descriptions for $A_i \uhr n$ and $B_j \uhr n$
from prefix-free conditional descriptions for $A_i \uhr n$ and $B_j \uhr n$,
given $n$ and $C(n)$ together with  a coding of these two parameters $n$ and $C(n)$.
Note that $n,C(n)$ can be given by a shortest
$C$-description of $n$ in one go, which needs $C(n)$ bits. With the
other two prefix-free conditional descriptions appended, one obtains for   a suitable   constant
$d'$ that %not GETTTTT- aaaaahrgggggg
\begin{center} $
   C(A_i \oplus B_j \uhr {2n}) \leq C(n)+K(A_i \uhr n|n,C(n))+
     K(B_j \uhr n|n,C(n))+ d'.
$\end{center}
Now one uses this to show the following as in (a):
\begin{eqnarray*}
    C(\mu \times \nu \uhr {2n})     & = & \sum_{i,j} \alpha_i \cdot \beta_j \cdot C(A_i \oplus B_j \uhr {2n}) \\
    & \leq & \sum_{i,j} \alpha_i \cdot \beta_j \cdot
              (C(n) + K(A_i \uhr n | n,C(n))+K(B_j \uhr n | n,C(n))+d') \\
    & \leq & \sum_{i,j} \alpha_i \cdot \beta_j \cdot (C(n)\!+\!2\!\cdot\!
      ((C(A_i \uhr n)\!-\!C(n))\!+\!(C(B_j \uhr n)\!-\!C(n)))\!+\!2d\!+\!d')\\
    & = & C(n)+2d+d' + 2 \cdot (\sum_i \alpha_i \cdot (C(A_i \uhr n)-C(n)))
          \cdot (\sum_j \beta_j)\\
    & & \mbox{ } \ \ \ \ \ \ \ \ + \ 2 \cdot (\sum_i \alpha_i) \cdot
                         (\sum_j \beta_j \cdot (C(B_j \uhr n)-C(n))) \\
    & \leq & C(n)+4c+2d+d'.
\end{eqnarray*}
The inequality from the first to the
second line follows from the definition of $d'$; the inequality from  the second  to the  third line follows
from Fact~\ref{fa:upgrade}.
%, in the form that  if $b$ is chosen least such that $K(x|n,C(n)) \leq 2b+d$, then $C(x)-C(n) \ge b$. 

%that is, from the existence of a $d$ with 
%$$
%   \forall n\,\forall b\,\forall x\in \{0,1\}^n \mbox{ with } C(x) \leq C(n)+b
%   \ [K(x|n,C(n)) \leq 2b+d];
%$$
  Thereafter,  one uses the distributivity of absolutely converging sums, and the definition of $c$.
%  , and 
%in the final  equality the bounds on these sums for the last
%inequality to the last line of the series of inequalities.
The same can be proven for $2n+1$ with only a slight increase in  notational
complexity. %Thus $\mu \times \nu$ is $C$-trivial. Why repeat?? Old fashioned
\end{proof}

The next proposition shows that three conditions involving
the descriptive complexities of strings are equivalent.
It is unknown whether these equivalent conditions are true.
These conditions relate fundamental properties of plain
and prefix-free Kolmogorov complexity;  properties
of similar type have been investigated quite a lot in
algorithmic information theory; fairly  recent sample references are
\cite{Bauwens:thesis,Bauwens:16,Miller:11,Li.Vitanyi:book}. The last condition states informally that there are finitely many potential ways to compute $C(n)$ from $n$ and $K(n)$, and for each $n$ one of them succeeds. (In fact at present it is unknown whether there is a single way.)
%Chaitin:75,Chaitin:76,Levin:73,,Solovay:75 overkill
\begin{proposition} \label{pr:equivalence}
The following   conditions are equivalent:
\begin{enumerate}[\bf(a)]
\item There is a constant $c$ such that for all $n$ and all $x \in \{0,1\}^n$,

\n
      $K(x)-K(n) \leq 2 \cdot (C(x)-C(n))+c$;
\item For each $c'$ there is  $d'$ such that for all $n$ and all $x \in \{0,1\}^n$, 

\n 
  if $C(x)-C(n) \leq c'$ then $K(x)-K(n) \leq d'$;
\item There is a constant $c''$ with $K(C(n)|n,K(n)) \leq c''$ for all $n$.
\end{enumerate}
\end{proposition}

\begin{proof}
The implication (a) to (b) is straightforward: if $c'$ bounds
$C(x)-C(n)$ then $d' = 2c'+c$ bounds $K(x)-K(n)$.

The implication (b) to (c) can be seen as follows: Note that for almost
all $n$ it holds that $C(n) < n$. So given $n$ sufficiently large,
one considers the string $x_n = (\{C(n)\} \uhr n)$; this string has
plain Kolmogorov complexity satisfying $C(x_n) =^+ C(n)$, let $c'$
be such that $C(x_n) \leq C(n)+c'$. Thus $K(x_n) \leq K(n)+d'$
for all $n$ where $x_n$ is defined by condition (b). By
\cite[Theorem 2.2.26]{Nies:book} there are at most $O(2^{d'})$
strings $y \in \{0,1\}^n$ with $K(y) \leq K(n)+d'$ and one can
enumerate the first, the second, $\ldots$, the $k$-th until
$x_n$ comes up, let $k$ be the corresponding sequence-number.
Now $k \leq 2^{d''}$ for some constant independent of $n$
and therefore $K(x_n|n,K(n)) \leq^+ d''$. As one can compute
$C(n)$ from $x_n$, there is a further constant $c''$
with $K(C(n)|n,K(n)) \leq c''$ for all $n$.

The implication (c) to (a) can be derived as in 
Theorem~\ref{pr:sixfive}~(b): for a
string $x$ of length $n$ with $C(x)-C(n) \leq b$
one has  $K(x|n,C(n)) \leq^+ 2b$. Now by condition (c)
it holds that $K(C(n)|n,K(n))$ is bounded by a constant; 
  hence so is    $K(x|n,K(n)) \leq^+ 2 \cdot (C(x)-C(n))$.
This yields condition (a).
\end{proof}

\n
If these conditions are true then the question whether  all $C$-trivial measures
are $K$-trivial has an affirmative answer.  To see this note that if $\mu = \sum_r \alpha_r \cdot \delta_{A_r}$ is
$C$-trivial then there is a constant $b$ such that
$\sum_r \alpha_r \cdot (C(A_r \uhr n)-C(n)) \leq b$ for all $n$.
It follows from Condition~(a) that
\begin{eqnarray*}
  K(\mu \uhr n) & = &  K(n)+\sum_r \alpha_r \cdot (K(A_r \uhr n)-K(n)) \\
  & \leq & K(n) + \sum_r \alpha_r \cdot (2 \cdot (C(A_r \uhr n)-C(n))+c) \\
  & \leq & K(n) + c + 2 \cdot \sum_r \alpha_r \cdot (C(A_r \uhr n)-C(n))
    \leq K(n)+2b+c.
\end{eqnarray*}

\medskip
\n
{\em Strong $C$- and strong $K$-triviality for measures.}

\n 
Given the examples above we consider a strengthening of our notions.

\begin{definition}
Call a measure $\mu$ {\em strongly $C$-trivial} if it can be written as a convex combination
$\sum_s \alpha_s \delta_{A_s}$ such that the $A_s$ are   $C$-trivial via 
constants $b_s$ such that  $\sum_s \alpha_s \cdot b_s $
is finite. Similarly,  one defines {\em strongly
$K$-trivial} measures.
\end{definition}

\begin{fact} \label{fa:convergentsum}
Every strongly $C$-trivial measure is $C$-trivial.  Every
strongly $K$-trivial measure is $K$-trivial.
\end{fact}

\begin{proof}
By hypothesis $\mu = \sum_s \alpha_s \cdot \delta_{A_s}$ where  each $A_s$
is  $C$-trivial with  constant $b_s$, and  $\sum \alpha_s \cdot b_s < c$. For all $n$ we have  $C(\mu \uhr n) = \sum_s \alpha_s \cdot
C(A_s \uhr n) \leq \sum_s \alpha_s \cdot (C(n)+b_s) \leq
C(n) + \sum_s \alpha_s \cdot b_s \leq C(n)+c$. The proof is  analogous for $K$.
\end{proof}

\n
While we ignore  whether every $C$-trivial measure is $K$-trivial, we show that 
this implication holds  for the strong versions of the two notions.
For this we prove a lemma about sets of interest on its own. Its proof  owes to the proof   that each $C$-trivial is recursive (Chaitin \cite{Chaitin:76},
also  see \cite[Proof of Thm 5.2.20(i)]{Nies:book}).

\begin{lemma}
There is an absolute constant $d$ such that every
 set that is $C$-trivial  via a  constant $b$
is $K$-trivial  via the constant $2b+d$.
\end{lemma}

\begin{proof}
As mentioned in the proof of Fact~\ref{fa:upgrade},
there are only $O(b^2 \cdot 2^b)$ binary strings $x$ of length $n$ with
$C(x) \leq C(n)+b$. Now one considers the following recursively enumerable
tree $T_b$ of binary strings: a string $x$ of length $n$ is in $T_b$
iff there is an $y$ extending $x$ of length $2^{n+1}$ such
that for all $m \leq 2^{n+1}$ it holds that $C(y \uhr m) \leq \log(m)+b+1$.

Note that there is an $m$ with $2^n \leq m \leq 2^{n+1}$ such that
$C(m) > n$. Thus for that length $m$, there are only $O(b^2 \cdot 2^b)$
strings which qualify and therefore, for length $n$, there are at most
$O(b^2 \cdot 2^b)$ many strings in $T_b$. Furthermore, note that the
membership in $T_b$ is closed under prefixes and therefore $T_b$
is a recursively enumerable tree.

For each length $n$, one can
enumerate from given $n,b$ the $O(b^2 \cdot 2^b)$ strings of that
length in $T_b$. So one can describe these strings by any description
of $n$ plus prefix-free codes of size $2b+d$ which compute the strings
given $n,b$. Note that these codes just say ``the first string of length
$n$ enumerated into $T_b$'', ``the second string of length $n$ enumerated
into $T_b$'' and so on. As there are $O(2^b \cdot b^2)$ of these strings,
one can represent them and $b$ in a prefix-free way with $2b$ plus
constant bits. Thus they can be described in a prefix-free way
with a description of size $K(n)+2b+d$ where $d$ is a constant
independent of $b,n$.
Hence $K(A \uhr n) \leq K(n)+2b+d$.
\end{proof}

\begin{proposition} \label{pr:strongimplication}
Every strongly $C$-trivial measure is strongly $K$-trivial.
\end{proposition}

\begin{proof}
Suppose that  a measure 
$\mu $ is given as a convex combination $ \sum_s \alpha_s \cdot \delta_{A_s}$
where each $A_s$ has the $C$-triviality constant $b_s$, and $\sum_s \alpha_s \cdot b_s$  is finite. 
 By the foregoing lemma every $A_s$ is  $K$-trivial via the  constant $d+2b_s$.
Clearly   $\sum_s \alpha_s \cdot (d+2b_s) =  d+2 \cdot \sum_s \alpha_s b_s$ is finite. So $\mu$ is strongly $K$-trivial.
\end{proof}

\n
One might ask whether every $K$-trivial measure is strongly $K$-trivial, 
and whether every $C$-trivial measure is strongly $C$-trivial. The answer to this
question is ``no'' as the following result shows.

\begin{proposition}
There is a measure $\mu$ which is $C$-trivial and $K$-trivial but not 
 strongly $K$-trivial (and hence not strongly $C$-trivial).
\end{proposition}

\begin{proof}
For each $n$, let $p_n$ be the descripikjtion of length up to $2^n$
for the prefix-free universal machine which produces
the longest output string of form $0^{m_n}$; here recall that $m_n$ and
$0^{m_n}$ are identified with each other.
Note that $K(p_n)$ and $|p_n|$ are both approximately $2^n$.
Now one defines the sets $A_n = \{m_n+k: 1p_{n+1}$ has a $1$ at position $k\}$.
The sets $A_n$ are uniformly limit recursive and the approximations converge
in time $m_{n+1}$ to $A_n$; furthermore, $\max(A_n) \leq m_n+2^{n+1}$.
Every set $A_n$ is finite and thus $K$-trivial.
Let $\mu = \sum_n 2^{-n-1} \cdot \delta_{A_n}$.

We show that $\mu$ is $K$-trivial.
One sees that for $\mu \uhr \ell$ with $m_n \leq \ell < m_{n+1}$,
$2^{-n-1}$ of the mass is concentrated on $\emptyset \uhr \ell$ and
for $h=0,1,\ldots,n$, $2^{-h-1}$ of the mass is concentrated on $A_h \uhr \ell$.
For $h < n$, the string $A_h \uhr \ell$ can be computed from $h$ and $\ell$.
Thus it has prefix-free complexity $K(\ell)+O(\log(h))$ which contributes to
the average measure $K(\mu \uhr \ell)$ a term bounded by
$2^{-h-1} \cdot K(\ell)+2^{-h-1} \cdot 2 \log(h)$.
As the sum $2^{-h-1} \cdot 2 \cdot \log(h)$ converges
to a constant $c$, one can estimate the overall sum as
$$K(\ell)+c+ 2^{-n-1} \cdot (K(A_n \uhr \ell)-K(\ell)),$$ and the latter
is bounded by $K(\ell)+c+2$, as $K(A_n \uhr \ell) \leq K(\ell)+2^{n+2}$.
Thus $\mu$ is $K$-trivial. By a similar argument, $\mu$ is  $C$-trivial.

The $K$-triviality constant $b_n$ for $A_n$ is bounded from below
by   $K(A_n \uhr {m_n+2^{n+1}})-K(m_n+2^{n+1})$ and that term is
at least $2^n$, as $K(p_{n+1}) =^+ 2^{n+1}$ and $K(m_n) =^+ 2^n$.
It follows that $\sum_n 2^{-n-1} b_n \geq \sum_n 1/2=\infty$. Thus $\mu$ is not strongly $K$-trivial.
\end{proof}

\n 
By the facts and propositions above, in Proposition~\ref{prop:gamma}
  the condition that $\Gamma$ be a
truth-table reduction  is necessary;
   downward closure of the $K$-trivial
measures does not hold for Turing reductions in general.

% which is only defined on the atoms of $\nu$ but
%not everywhere does.

\begin{proposition}
There are measures $\mu,\nu$ and a Turing reduction $\Gamma$
such that $\nu$ is strongly $K$-trivial,
$\Gamma$ is defined on all the atoms of $\nu$ and $\mu=\Gamma(\nu)$
is not $K$-trivial.
\end{proposition}

\begin{proof}
Consider the set $\Omega \oplus \overline \Omega$, that is, the
join of $\Omega$ with its complment. This set is left-r.e.\ and
initial segments of length $2m$ have approximately complexity $m$.
Now   let
$$
   A_m = \{x \in \Omega \oplus \overline \Omega: x < 2n_m\}
$$
and   $\gamma_m = \frac{1}{(m+1) \cdot (m+2)} = \frac{1}{m+1}-\frac{1}{m+2}$.
Furthermore, let $n_m$ be the least integer greater than  $(m+2)^{3/2}$ and
let $k_m = \langle m,c_{\Omega}(n_m)\rangle$ where $c_\Omega(n_m)$
is the time to follow the recursive left-enumeration of $\Omega$
until the first $n_m$ bits are correct (and therefore remain correct from then on).
Note that $k_m$ can be computed from $A_m \uhr {2n_m+2}$ and
$A_m \uhr {2n_m+2}$ can be computed from $k_m$. Let $B_m = \{k_m\}$
and $\nu = \sum_m \gamma_m \cdot \delta_{B_m}$ and
$\mu = \sum_m \gamma_m \cdot \delta_{A_m}$.

Note that for $\ell \geq k_m$, $k_m$ can be computed
from $\ell$ and $m$, as $\ell$ is an upper bound of $c_{\Omega}(n_m)$;
for $\ell < k_m$, $B_m \uhr \ell =0^\ell$. Thus the coding constant
of $B_m$ is $2\log(m)$ plus some constant.
As $\sum_m \gamma_m \cdot 2 \cdot \log(m)$ is a convergent sum -- almost
all terms are bounded by $(m+2)^{-3/2}$ which is a convergent sum
-- the measure $\nu$ is $K$-trivial by Fact~\ref{fa:convergentsum}.
Actually it is by definition even strongly $K$-trivial.

The measure $\mu$ is not $K$-trivial, as all $A_k$ with $k \geq m$
have the common prefix $u = \Omega \oplus \overline \Omega \uhr {2n_m}$
of length $2n_m$ and that prefix has, for almost all $m$, the Kolmogorov
complexity $(m+2)^{3/2}$ or more. Note that $\mu(\{u\} \times \cantor)$
is at least $\frac{1}{m+2}$, so this string contributes to the average
$\mu \uhr {2n_m}$ at least $(m+2)^{1/2}$. Thus $\mu$ is not $K$-trivial,
as the function $m \mapsto K(2n_m)$ cannot be bounded from below
by any increasing unbounded recursive function.

Now one defines  a Turing functional $\Gamma$ that  translates  each  $B_m$ to $A_m$:
the functional searches in the oracle $B_m$ for the first element $x$
of $B_m$; as long as the oracle is not empty, such an element is found.
Then it determines the $m,s$ with $x = \langle m,s\rangle$ and computes
the corresponding entry in the set
$\{x \in \Omega_s \oplus \overline \Omega_s: x < 2n_m\}$. In the case
that the oracle has the set $B_m$, the search gives $k_m$
and this Turing reduction provides the set $A_m$.
Thus $\Gamma$ has the desired properties. Note that $\Gamma$
is undefined  only on $\emptyset$,  and the measure of $\{\emptyset\}$ with
respect to $\nu$ is $0$.
\end{proof}

\n
While we know that every strongly $C$-trivial measure is strongly $K$-trivial,
we ignore  whether every $C$-trivial measure is $K$-trivial.
The converse implication fails even for Dirac measures because each
$C$-trivial sequence is computable by the aforementioned result of Chaitin.
We next  show that this failure of the converse implication
can   be witnessed by a strongly $K$-trivial measure where all the atoms are recursive,  and
thus $C$-trivial (as sequences).

\begin{proposition}
%:
There is a strongly $K$-trivial measure which is not $C$-trivial, and 
has only recursive atoms.
\end{proposition}

\begin{proof} Let $A$ be an r.e.\ $K$-trivial set which is not
$C$-trivial, that is, which is not recursive.
Let 
$a_{k,s}$be  the least number $m \leq s$ such that either
$m=s$ or $C_s(A_s \uhr m) \geq C_s(m)+4^k$; as $A$ is not $C$-trivial,
each $a_{k,s}$   converges to $a_k$,   the least value $m$
satisfying $C(A \uhr m) \geq C(m)+4^k$; note that from some stage on,
$A_s$ and $C_s$ have converged up to $a_k$ and from then onwards
$a_{k,s} = a_k$; thus for each $k$ the sequence of the $a_{k,s}$ is
uniformly recrusive and converges to $a_k$.

One  next defines a sequence $b_{k,s}$ such that $b_{k,0} = 0$
and $b_{k,s+1} \in \{b_{k,s},s+1\}$ where the value $s+1$ is taken
if either $a_{k,s} \not\leq b_{k,s}$
or there is $h<k$ with $b_{h,s} = b_{k,s}$. Note that the
definition of these values implies that $a_{k,s} \leq a_{h,s}$
and $b_{k,s} \leq b_{h,s}$ whenever $k \leq h$.

Now let $A_k = \{a \in A: a < b_k\}$. The set $A_k$ has the $K$-triviality
constant $2\log(k)$, as $A_k$ below $b_k$ equals the fixed
$K$-trivial set $A$ and for $s \geq b_k$, if one knows $k$ then
one can compute $b_k = b_{k,s}$ and $A_k \uhr s$ can then be computed
from $k$ and $A \uhr s$, as $s$ is the length of $A \uhr s$.

Having these ingredients, one chooses $\mu = \sum_k 2^{-k-1} \cdot \delta_{A_k}$. This measure is strongly $K$-trivial by Fact~\ref{fa:convergentsum}, since
the sum over $2^{-k-1} \cdot 2 \log(k)$ converges.
However, $\mu$ is not $C$-trivial, as $\mu \uhr {a_k}$ satisfies for $h \geq k$
that $A_h \uhr {a_k} = A \uhr {a_k}$ and
$$
   C(\mu \uhr {a_k}) \geq C(a_k) + 2^{-k} \cdot (C(A \uhr {a_k})-C(a_k))
      \geq C(a_k)+2^k,
$$
and the values $2^k$ are obviously not bounded by a constant.
\end{proof}

\n

\section{Full \ML\ randomness of measures} \label{s:full ML}

\n Let  $\+ M(\cantor)$ be the space of probability measures on Cantor
space (which is canonically a compact topological space). 
Probability measures on this space have been defined by 
   Mauldin and Monticino \cite{Mauldin.Monticino:95}. A special case it the uniform measure $\mathbb P$ on $\+ M(\cantor)$.

To define $\mathbb P$, first let $\+ R $ be the closed set of
representations of probability measures; namely, $\+ R$ consists of
the  functions $X\colon \{0,1\}^* \to [0,1]$ such that $X_\ES =1$ and
$X_\sss = X_{\sss0} + X_{\sss 1}$ for each string $\sss$. $\mathbb P$
is the unique measure on $\+ R$ such that for each string $\sss$ and
$r, s \in [0,1]$, we have 
  $\mathbb P(X_{\sss0} \le r \mid X_\sss = s) = \min (1, r/s)$.   
Intuitively, we choose $X_{\sss 0}$ at random w.r.t. the uniformly
distribution on the interval $[0, X_\sss]$, and the choices made at
different strings are independent.  In the language of~\cite{Mauldin.Monticino:95}, the transition kernel $\tau$ takes the value $\leb $ for each dyadic rational (corresponding to a string $\sss$).

We remark that   for $n>0$ and a string $\sss$ of length $n$,  the function $v_n$ on $[0,1]$  given by $x \to \mathbb P(\{\mu[\sss] \le x\})$ only depends on $n= \sssl$. In fact, the $v_n$ are given by the recursion $v_1 (x) = x$ and $v_{n+1}(x) = x  + \int_x^1 v_n (x/t)dt$. E.g., $v_2(x) = x(1-\ln x)$.

Culver's thesis~\cite{Culver:15} shows that this measure is
computable in the sense of Hoyrup and Rojas~\cite{Hoyrup.Rojas:09}. So the framework provided in~\cite{Hoyrup.Rojas:09} yields a
notion of \ML\ randomness for points in the space $\+ M(\cantor)$. 

\begin{proposition} \label{fffuuukkk}Every probability measure $\mu $
that is \ML\ random wrt $\mathbb P$ is \ML\ absolutely continuous.
\end{proposition}

\n For the duration of this proof  let $\mu$ range over ${\+ M
(\cantor)}$.  For an open set  $G \sub \cantor $,   let   \[r_G= \int
 \mu(G) d \mathbb P(\mu).\]  
Our  proof of Proposition \ref{fffuuukkk}  is based on    two facts.    

\begin{fact}  $r_G = \leb (G)$. \end{fact}

\begin{proof} Clearly,   for each $n$ we have \[\sum_{\sssl =
n}r_{[\sss]} = \int \sum_{\sssl = n} \mu ([\sss]) d\mathbb P(\mu)
=1.\] Furthermore,  $r_\sss = r_\eta$ whenever $\sssl = |\eta|=n$
because by the remark above there is a $\mathbb P$-preserving transformation $T$ of $\+
M(\cantor)$ such that $\mu([\sss]) = T(\mu)([\eta])$. Therefore
$r_{[\sss]} = \tp{-\sssl}$. 

If $\sss, \eta $ are incompatible then $r_{[\sss] \cup [\eta]} =
r_{[\sss]} + r_{[\eta]}$. Now it suffices to write $G = \bigcup_i
[\sss_i]$ where the strings $\sss_i$ are incompatible, so that $\leb G
= \sum_i \tp{-|\sss_i|}$. 
\end{proof}

\begin{fact} Let $\mu \in \+ M(\cantor)$ and let  $\seq{G_m}\sN m$ be
a  \ML-test  such that there is $\delta \in \QQ^+$ with $\fa m \, 
[\mu(G_m) > \delta]$. Then $\mu $ is not \ML-random w.r.t.\ $\mathbb P$. 
\end{fact}

\begin{proof} Observe that by the foregoing fact   \[\delta \cdot 
\mathbb P(\{ \mu \colon \mu(G_m) \ge  \delta\})  \le \int \mu(G_m) d
\mathbb P(\mu)  = \leb(G_m) \le \tp{-m}.\] 
Let  $\mathcal G_m = \{ \mu \colon \mu(G_m) >  \delta\}$ which is
uniformly effectively open in the space of measures  $\+ M(\cantor)$. 
Fix $k$ such that $\tp{-k} \le
\delta$. By the inequality above, we have $\mathbb P(\mathcal G_m) \le
\tp{-m}/\delta \le \tp{-m+k}$. Hence   $\seq {\mathcal G_{m+k}} \sN m$
is a \ML-test w.r.t.\
$\mathbb P$ that succeeds on $\mu$. 
\end{proof}

\n  This argument also works for   randomness notions stronger than
\ML's. For instance, if there is a weak-2 test  $\seq{G_m}\sN m$ such
that $\mu G_m > \delta $ for each $m$, then $\mu$ is not weakly
2-random with respect to $\mathbb P$. The converse
of Proposition \ref{fffuuukkk} fails. Culver~\cite{Culver:15} shows that
each measure $\mu$ that is \ML-random w.r.t.\ $\mathbb P$ is non-atomic.
So a measure $\delta_Z$ for a \ML-random bit sequences $Z$ is \MLac\ 
but not \ML-random with respect to~$\mathbb P$.

\section{Being ML-a.c.\ relative to computable ergodic measures} \label{s:SMB}

\n
We review some notions from the field of symbolic dynamics, a
mathematical area closely related to Shannon information  theory. See See e.g.~\cite{Shields:96}  for more detail. Thereafter we
will consider effective ``almost-everywhere theorems'' related to that
area in the framework of randomness for measures.

 In symbolic dynamics it can be  useful to
admit alphabets other than the binary one. Let     $\mathbb A^\infty$
denote the topological space of one-sided infinite sequences of
symbols in an alphabet $\mathbb A$.  Randomness notions etc.\ carry
over from the case of $\mathbb A = \{0, 1\}$. 
A~dynamics on $\mathbb A^\infty$  is given by the shift operator $T$,
which erases the first symbol of   a sequence.
A measure $\rho$ on $\mathbb A^\infty$ is called  \emph{shift 
invariant} if $\rho ( G) = \rho (T^{-1}(G))$ for each open   (and
hence each measurable) set $G$. 
The \emph{empirical entropy}  of a measure  $\rho$ along $Z\in \mathbb
A^\infty$ is given by the sequence of   random variables  \[
h^\rho_n(Z ) = -\frac 1 n \log_{|\mathbb A|} \rho [Z\uhr n].\]   
A shift invariant measure $\rho$ on $\mathbb A^\infty$  is called 
\emph{ergodic} if   every  $\rho$ integrable function $f $ with $f
\circ T = f$  is constant $\rho$-almost surely.
The following  equivalent  condition can be easier to check: for any strings
 $u,v \in \mathbb A^*$,  
\[ \lim_N \frac 1 N \sum_{k=0}^{n-1} \rho ([u] \cap T^{-k}[v]) = \rho
[u] \rho [v]. \]
For ergodic $\rho$, the entropy $H(\rho)$ is defined as $\lim_n
H_n(\rho)$, where \[H_n(\rho) =  -\frac 1 n \sum_{|w| = n}  \rho [w]
\log \rho [w].\]
Thus, $H_n(\rho) = \mathbb E_\rho  h^\rho_n$ is the expected value
with respect to $\rho$.  
by concavity of the logarithm function   the limit 
exists and equals the infimum of the sequence. This limit is denoted
$H(\rho)$, the entropy of $\rho$.

A  well-known result from the 1950s due to Shannon, McMillan and
Breiman  states that  for an  ergodic
measure $\rho $, for $\rho$-a.e.\ $Z$  the empirical  entropy along
$Z$ converges to  the entropy of the measure. See e.g.~\cite[Chapter
1]{Shields:96}, but note that the result is called the Entropy Theorem there. 

\begin{theorem}[SMB   theorem] Let  $\rho $ be
an  ergodic   measure 
on the space $\mathbb A^\infty$. 
For $\rho$-almost every~$Z$   we have  $\lim_n h^\rho_n(Z) = H(\rho)$.
\end{theorem}

\n
Recall from 
Fact \ref{fa:univ}  that a measure $\mu$ is \emph{\MLac\
with respect to $\rho$} iff $\mu(\+ C) = 0$ where $\+ C$ is the class
of sequences in $\mathbb A^\infty$ that are not \ML-random with respect
to $\rho$. 

If a computable measure $\rho$ is shift invariant, then   $\lim_n
h^\rho_n(Z)$  exists for each $\rho$-\ML-random~$Z$ by  a result of
Hochman \cite{Hochman:09}. Hoyrup \cite[Theorem 1.2]{Hoyrup:12}
gave   an alternative proof for ergodic $\rho$, and also showed that
in that case  we have $\lim_n h^\rho_n(Z) = H(\rho)$ for each
$\rho$-ML-random $Z$.    We extend this result to measures $\mu$  that
are \MLac\ with respect to $\rho$, under the additional  hypothesis that
 the $h_n^\rho$  are uniformly bounded. This holds e.g.\ for Bernoulli
measures and the measures given by a Markov process. 

\begin{proposition} \label{prop: bound}
Let  $\rho $ be a computable  ergodic        measure
on the space $\mathbb  A^\infty$ such that for some constant $D$,  
each $h_n^\rho$ is   bounded above by $D$. Suppose the measure $\mu$
is \MLac\  with respect to $\rho$.  Write $s = H(\rho)$. 
We have $\lim_n E_\mu |h^\rho_n - s| =0$. 
\end{proposition} 

\begin{proof}
By   Hoyrup's result, $\lim_n h^\rho_n(Z) = s$ for each
$\rho$-ML-random $Z$. Since the sequences that are not \ML-random
w.r.t.\ $\rho$ form a null set w.r.t.\ $\mu$,  we infer   that 
$\lim_n | h^\rho_n(Z) - s| =0$ for  $ \mu$-a.e.\ $Z$. Since the 
exception set  is  measurable and the $h_n\rho$ are bounded,
the Dominated Convergence Theorem now shows that  $\lim_n \mathbb
E_\mu |h_n^\rho - s| =0$, as required.
\end{proof}

\n
We now given an example showing that the boundedness hypothesis on the
$h_n^\rho$ is necessary. We provide a computable ergodic measure
$\rho$ such that some finite measure $\mu \ll \rho$ makes the sequence
$E_\mu h_n^\rho$ converge to $\infty$. 
\begin{proposition} There is an ergodic computable measure $\rho$
(associated to a binary renewal process) and a computable measure $\mu
\ll \rho$ such that 
$\lim_n \mathbb E_\mu h^\rho_n = \infty$. (We can then normalise $\mu$
to become a probability measure, while maintaining the same conclusion.)
\end{proposition}

\begin{proof} Let $k$ range over positive natural numbers.
The real  $c = \sum_k \tp{-k^4}$ is computable. Let $p_k = 
\tp{-k^4}/c$ so that $\sum p_k=1$.   Let $b= \sum_k k \cdot p_k$ which
is also computable.

Let $\rho $ be the measure associated with the corresponding binary
renewal process, which is  given by the conditions \begin{center}
$\rho[Z_0 =1]= 1/b$ and $\rho(10^k 1 \prec Z  \mid Z_0 =1) = p_k$.  \end{center}
Informally, the process has initial value  $1$ with probability $1/b$,
and after each $1$,  with probability $p_k$ it takes $k$ many 0s until
it reaches the next 1. See again e.g.~\cite[Chapter 1]{Shields:96} where
it is shown that $\rho$ is ergodic.
Write $v_k = 1 0^k 1$. Note that $\rho[v_k] = p_k/b$. 

Define a  function $f$ in $L_1(\rho)$ by $f(v_k \ape Z ) =  
k^{-2}/p_k$ and $f(X)=0$ for any $X$ not extending any $v_k$.
  It is clear that $f$ is $L_1(\rho)$-computable, in the usual sense
that there is an effective sequence of basic functions $\seq{f_n}$
converging effectively to $f$: let $f_n(X)= f(X)$ in case $v_k \prec
X$, $k \le n$, and $f_n(X)= 0 $ otherwise. Define   the    measure
$\mu$   by  $d\mu = f d\rho$, i.e. $\mu(A) = \int_A f d\rho$.  Thus
$\mu[v_k]= k^{-2}/b$. Since $\rho$ is computable and $f$ is
$L_1(\rho)$-computable, $\mu$ is computable. Also note that
$\mu(\cantor)= \int f d\rho$ is finite.

For any $n>2$, letting $k= n-2$, we have
\[E_\mu h_n^\rho \ge -\frac 1 n \mu[v_k] \log \rho[v_k]= - \frac 1 {n
 k^{2} b}   (k^4-bc) \ge \frac {k^2} {nb} -O(1).\]
This completes the proof.
\end{proof}

\noindent
The next observation shows that the asymptotic initial segment
complexity of a \MLac\ measure relative to $\rho$ obeys  some lower
bound. Note that $H(\leb) = 1$. So for $\rho = \leb$, this shows that
in Example~\ref{ex:slow growth} we cannot subtract, say,  $ n/4$
instead~of~$n^\theta$.  

\begin{proposition} \label{prop: dim =1}
Let $\rho$ be a computable ergodic   measure,
and suppose  $\mu$ is a \MLac\
measure with respect to $\rho$.
Then $$\lim_n \frac 1 n K(\mu\uhr n) = \lim_n \frac 1 n C(\mu\uhr n) =
H(\rho).$$ 
\end{proposition}

\begin{proof}
We can use  $K$ and  $C$ interchangeably
because  $C(x) \lep K(x) \lep C(x)  + K(C(x))$
\cite[Proposition 2.4.1]{Nies:book}. We settle on $K$.

Let $k_n(Z) = K(Z\uhr n)  /n$. 
The argument   is very similar to the
one in Proposition~\ref{prop: bound} above, replacing the functions
$h_n$ by the $k_n$. Note that $k_n $ is bounded above
by a constant because $K(x) \lep |x| +  2 \log |x|$. 
Hoyrup's result \cite[Theorem 1.2]{Hoyrup:12} states that $\lim_n k_n(Z)
= H(\rho)$ for each $\rho$-ML-random $Z$.  Now we can apply
the Dominated Convergence Theorem as in the proof of 
Proposition~\ref{prop: bound}.
\end{proof}

\section{Conclusions and open questions}

\noindent
In the present paper, we studied
algorithmic randomness, and triviality properties for probability measures on
Cantor space.

A main property we studied was  that  a  measure $\mu$ is \MLac:   for every   \ML\ test  the $\mu$-measure of
its components converges to $0$.
%We think of this notion as a weak randomness notion for measures.
We provided several examples and showed a robustness property related
to Solovay tests.
In Section~\ref{s:sS} we introduced
  strong Solovay tests for measures.  We
leave the following question open.

\begin{question}
Does every measure that passes every \emph{strong} Solovay test,   
pass every Solovay test?
\end{question}

\n
Thereafter, we studied growth behaviour of the $\mu$-averages of the descriptive string  complexities
$C(x)$ and $K(x)$ over strings $x$ of length $n$.
We asked what can be said for measures where the
growth is as slow  as possible and for those where it is as fast 
as possible. We looked at $C$- and at $K$-trivial measures, where the averages
are bounded up to a constant by $C(n)$ or $K(n)$, respectively.
Such measures are atomic, that is, they are   weighted, possibly infinite
sums of Dirac measures. The atoms of such  measure satisfy the corresponding
triviality property; furthermore, the weighted sum of the atoms has to
converge fast, though we could not find a precise criterion for this.
In this context, we mention two open questions.

\begin{question} \label{qu:cktrivial}
Is every $C$-trivial measure also $K$-trivial?
\end{question}

\n
An affirmative answer to the next question  would lead to an affirmative answer to
Question~\ref{qu:cktrivial};
  Proposition~\ref{pr:equivalence} provided more detail.

\begin{question} \label{qu:bound CK}
Is there a constant bound on the function $n \mapsto K(C(n)|n,K(n))$?
\end{question}

\iffalse \n
Note that if  for an atomic measure one choses each time
finitely many atoms to be the same, then one can achieve an exponential
rate of convergence, which allows to choose the corresponding sets such
that the measure is $C$-trivial. However, it is not clear whether this
works in general with a one-one assignment of atoms. 
%Makes no sense. We shouldn't list failed attempts as open questions

\begin{question}
Given a sequence of positive numbers $\gamma_k$ such that $\sum_k \gamma_k=1$, 
can one choose sets $A_k$ or $B_k$ such that $\sum_k \gamma_k \cdot A_k$
is C-trivial and $\sum_k \gamma_k B_k$ is $K$-trivial?
\end{question}
\fi

\n Measures $\mu$ on the upper end of the growth 
spectrum are \MLac.  However, we showed that also here there is no
exact correspondence between the growth-rate of the $\mu$ averages of the string complexity and the
property of being \MLac; instead,  in the border area, there are 
measures that are smaller with respect to the growth-rate of the averages which   are
\MLac\, and  measures that are larger but which are not. In particular, an analogue of
the Levin-Schnorr Theorem does not hold for measures. We noted that near the 
upper end of the growth spectrum for $C$ are the  Kolmogorov random measures $\mu$:  there is a constant $d$ with $C(\mu \uhr n) \geq n-d$
for infinitely many~$n$. Besides the usual uniform measure, there are
  other examples, such as  Dirac measure of sets which are 2-random.
While $C$-trivial and $K$-trivial measures are closed under finite
convex combinations, this is not known for the measures on the upper end of the growth spectrum.

\begin{question}
Assume that measure  $\mu,\nu$  are Kolmogorov random.
Is  every convex
combination of $\mu$ and $\nu$ Kolmogorov random as well?
\end{question}

\n
As mentioned in  Section~\ref{s:insegs growth},  a sequence $X$ is
2-random iff it is     Kolmogorov
randomness as defined in \cite{Li.Vitanyi:book}).
 So an affirmative answer to the question above implies an affirmative answer to the following:
\begin{question}
If $X,Y\in \cantor $ are both 2-random, does there exist a constant
$d$ such that infinitely many $n$ satisfy both
$C(X \uhr n) \geq n-d$ and $C(Y \uhr n) \geq n-d$?
\end{question}

\n
Besides the trivial case when $X=Y$, this also holds if $X \oplus Y$ is
2-random. This can be seen using that 2-randomness  coincides with 
Kolmogorov randomness. Clearly for strings $x,y$ of length $n$ one has $C(x \oplus y) \lep C(x) + n$: on input $\tau$ a plain machine $M$ searches for a decomposition $\tau = \sss \ape y$ such that the universal plain machine on input $\sss$ outputs a string $x$ of the same length as $y$; then $M$ outputs $x \oplus y$. So  if   $C(X \oplus Y \uhr {2n}) \gep 2n$ for infinitely
many $n$, then $C(X\uhr n)  \gep n$ for the same witnesses~$n$. By symmetry, also $C(Y\uhr n)  \gep n$ for the same witnesses $n$.

 The following diagram summarizes the implications obtained between (weak) randomness notions for measures. Non-labelled implication arrows are trivial. We ignore at present whether the implications given by the arrows labelled  \ref{fact:CK imply} and \ref{fact: compare CK} are proper. All the other implication arrows are proper. In most cases this can be seen already for  Dirac measures. For instance, full Martin-L\"of randomness of a measure implies that the measure has no atoms, so the converse of the implication labeled Prop.\ \ref{fffuuukkk} fails.  We can rule out two  further implications.  Each fully ML-random $\mu$ is orthogonal to $\leb$ by Bienvenu and Culver~\cite[Section~2.6]{Culver:15}, and hence not absolutely continuous. We ignore whether full ML-randomness implies  one of the  condition stronger than ML-a.c.\  displayed in the lower part of the diagram.

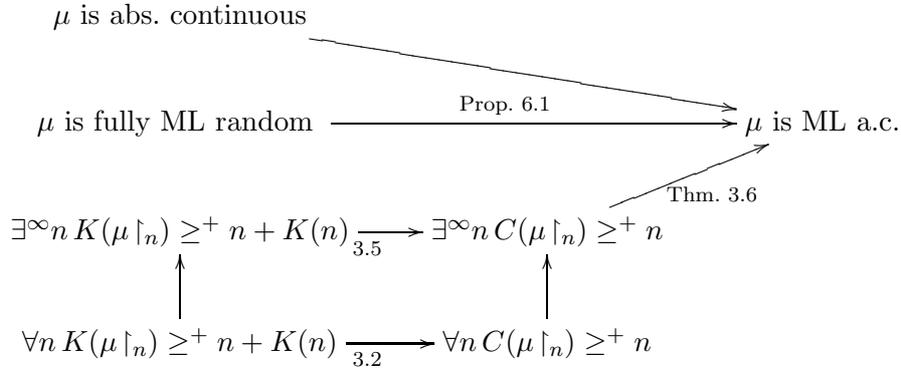
\[\xymatrix{\mu \text{ is abs.\ continuous} \ar[rrd]  & & \\ 
\mu \text{ is fully ML random  } \ar [rr]^{\text{ Prop.\ \ref{fffuuukkk}}} & & \mu \text{ is ML a.c.} \\
\ex^\infty n \, K(\mu\uhr n ) \gep n+K(n) \ar[r]_{\text{   \ref{fact:CK imply}}} & \ex^\infty n \, C( \mu\uhr n)  \gep n \ar[ru]_{\text{ Thm.\ \ref{thm:Inseg MLR}}}  & \\
\fa n \, K(\mu\uhr n ) \gep n+K(n)\ar[u] \ar[r]_{\text{   \ref{fact: compare CK}}} & \fa n \,  C( \mu\uhr n)  \gep n \ar[u]   &  
}\]

In Section~\ref{s:SMB} we considered weak randomness relative to a general
ergodic computable measure. We proved  appropriate effective
versions of the Shannon-McMillan-Breiman theorem and the Brudno
theorem where the bit sequences are replaced by measures; the former needed an additional boundedness hypothesis.

One possible  effective version of the Birkhoff ergodic theorem states that  for an ergodic
computable $\rho$, if $f \cantor \to \R$ is $\rho$-integrable and
lower semicomputable and $Z$ is $\rho$-\ML-random, then  the limit of
the usual ergodic averages $A_nf(Z) = \frac 1 n \sum_{k< n} (f \circ
T^k)(Z)$ equals $\int f d\rho$. For background see e.g.\
\cite{Miyabe.Nies.Zhang:16} which contains references to original
work.  If in addition the $A_nf$ are bounded then an argument similar to the one
in the proof of Proposition~\ref{prop: bound} shows that
$\lim_n\int (A_nf) d\mu = \int fd\rho$ for any measure 
$\mu\ll_{ML} \rho$. However, it is unknown whether
  this additional hypothesis is necessary.

\smallskip
\n
{\bf Acknowledgements.}
The authors would like to thank the anonymous referees of STACS 2020
for useful suggestions. Furthermore, they would like to thank Benjamin Weiss and Bruno
Bauwens for helpful correspondence.   

 Some problems from this  paper were presented  at  the 2020 AIM workshop on randomness and applications, for instance Question~\ref{qu:bound CK}. The authors thank the participants, in particular J.\ Miller, S.\ Shen and E.\ Mayordomo,  for their interest and useful comments. 

\def\cprime{$'$} \def\cprime{$'$}

\end{document}